\documentclass[11pt,reqno]{amsart}
\usepackage{amsmath,amssymb,amsthm,amsfonts,latexsym, amsthm, amscd, mathrsfs, stmaryrd}
\usepackage[colorlinks=true,linkcolor=blue,citecolor=blue]{hyperref}
\usepackage{graphicx,color,colortbl}
\usepackage{upgreek}
\usepackage[all,cmtip]{xy}
\usepackage[mathscr]{euscript}
\usepackage{hhline}
\numberwithin{equation}{section}
\usepackage{dynkin-diagrams}
\usepackage{epic}
\allowdisplaybreaks
\numberwithin{equation}{section}
\usepackage{float}
\usepackage{tikz}
\usetikzlibrary{
    cd,
	shapes,
	arrows,
	positioning,
	decorations.markings,
	decorations.pathmorphing,
	circuits.logic.US,
	circuits.logic.IEC,
	fit,
	calc,
	plotmarks,
	matrix
}

\usepackage{enumerate}

\newtheorem{proposition}{Proposition}[section]
\newtheorem{lemma}[proposition]{Lemma}
\newtheorem{corollary}[proposition]{Corollary}
\newtheorem{theorem}[proposition]{Theorem}

\newtheorem{innercustomthm}{{\bf Theorem}}
\newenvironment{customthm}[1]
  {\renewcommand\theinnercustomthm{#1}\innercustomthm}
  {\endinnercustomthm}
\theoremstyle{definition}
\newtheorem{definition}[proposition]{Definition}

\theoremstyle{remark}
\newtheorem{remark}[proposition]{Remark}

\newcommand{\arxiv}[1]{\href{http://arxiv.org/abs/#1}{\tt arXiv:\nolinkurl{#1}}}

\newcommand{\Rmnum}[1]{\expandafter\@slowromancap\romannumeral #1@}
\def \g{\mathfrak{g}}

\def \dm{\diamond}

\def \N{\mathbb{N}}
\def \Q{\mathbb{Q}}

\def \Z{\mathbb{Z}}
\def \I{\mathbb{I}}
\def \Br{\mathrm{Br}}

\def \bP{\mathbf{P}}
\def \bF{\mathbb{F}}

\def \bc {\mathbf{c}}
\def \bd {\mathbf{d}}
\def \fg{\mathfrak{g}}
\def \A{\mathcal{A}}

\def \fX{\Upsilon}
\def \wI{\I_{\circ}}
\def \wItau{\I_{\circ,\tau}}
\def \bI{\I_{\bullet}}

\def \diag{\mathrm{diag}}
\def \gr{\mathrm{gr}}
\def \tfX{\widetilde{\Upsilon}}
\def \cR{\mathcal{R}}
\def \bs{\mathbf{r}} 
\def \bF{\mathbb{F}}
\def \bw{w_\bullet}
\def \bwi{w_{\bullet,i}}
\def \bW{W_{\bullet}}

\def \bbw{{\boldsymbol{w}}}
\def \tbU{\tU_{\bullet}}

\def \ba{\mathbf{a}}
\def \tk{k}
\def \tT{\widetilde{\mathscr T}}
\def \tTD{\widetilde{T}}

\def \bT{\mathbf{T}}

\def \Id{\mathrm{Id}}

\newcommand{\U}{\mathbf{U}}
\newcommand{\tK}{K}
\newcommand{\tU}{\widetilde{{\mathbf U}} }
\newcommand{\tUi}{\widetilde{{\mathbf U}}^\imath}

\newcommand{\Aut}{\operatorname{Aut}\nolimits}

\newcommand{\qbinom}[2]{\begin{bmatrix} #1\\#2 \end{bmatrix} }

\def \ov{\overline}
\def \un{\underline}

\newcommand{\nc}{\newcommand}
\nc{\greentext}[1]{\textcolor{green}{#1}}
\nc{\redtext}[1]{\textcolor{red}{#1}}
\nc{\bluetext}[1]{\textcolor{blue}{#1}}
\nc{\brown}[1]{\browntext{ #1}}
\nc{\green}[1]{\greentext{ #1}}
\nc{\red}[1]{\redtext{ #1}}
\nc{\blue}[1]{\bluetext{ #1}}

\def \Hom{\mathrm{Hom}}
\def \Q {\mathbb Q}
\def \TT{\widetilde{\mathbf T}}

\def \bvs{{\boldsymbol{\varsigma}}}

\def \vs{\varsigma}
\def \U{\mathbf U}
\def \Ui{\mathbf{U}^\imath}
\def \reW{W^\circ}

\newcommand{\ev}{\bar{0}}
\newcommand{\odd}{\bar{1}}

\usepackage{latexsym,todonotes}

\begin{document}

\title[PBW bases for $\imath$quantum groups]{PBW bases for $\imath$quantum groups
}

	\author[Ming Lu]{Ming Lu}
	\address{Department of Mathematics, Sichuan University, Chengdu 610064, P.R.China}
	\email{luming@scu.edu.cn}

\author[Ruiqi Yang]{Ruiqi Yang}
\address{Department of Mathematics, Sichuan University, Chengdu 610064, P.R.China}
\email{yangruiqi@stu.scu.edu.cn}

	\author[Weinan Zhang]{Weinan Zhang}
	\address{Department of Mathematics and New Cornerstone
Science Laboratory, The University of Hong Kong, Hong Kong SAR, P.R.China}
	\email{mathzwn@hku.hk}



\begin{abstract}
We establish PBW type bases for $\imath$quantum groups of arbitrary finite type, using the relative braid group symmetries.  Explicit formulas for root vectors are provided for $\imath$quantum groups of each rank 1 type. We show that our PBW type bases give rise to integral bases for the modified $\imath$quantum groups. The leading terms of our bases can be identified with the usual PBW bases in the theory of quantum groups. 
\end{abstract}

\maketitle

\setcounter{tocdepth}{1}

\tableofcontents

\section{Introduction}

\subsection{Backgrounds}


Introduced by Lusztig, the PBW bases have played an essential role in the theory of quantum groups $\U$. These bases are constructed using the braid group symmetries (see \cite[Chapter 8]{Jan96}), and have applications in the algebraic construction of the canonical bases \cite{Lus90b,Lus91}. Let $\A=\Z[q,q^{-1}]$ and $\U^-_\A$ be Lusztig's integral form for the half $\U^-$ of $\U$ \cite[3.1.13]{Lus93}. Lusztig showed that the PBW bases are integral bases for $\U^-_\A$, and give rise to integral bases for the modified quantum group $\dot{\U}$ \cite{Lus93}.

Let $\theta$ be an involution on a Lie algebra $\g$. It is well known that irreducible symmetric pairs $(\fg,\fg^\theta)$ are classified by Satake diagrams $(\I =\I_\circ \cup \I_\bullet, \tau)$; here $\I =\I_\circ \cup \I_\bullet$ is a partition of vertices in the underlying Dynkin diagram and $\tau$ is a diagram involution.
As a quantization of symmetric pairs, Letzter \cite{Le99} introduced quantum symmetric pairs ($\U$,$\Ui$).  
The $\imath$quantum groups $\Ui=\Ui_\bvs$ are right coideal subalgebras of $\U$ defined using the Satake datum depending on parameters $\bvs=(\varsigma_i)_{i\in\I_\circ}$.  
Motivated by the Hall algebra realization of $\imath$quantum groups \cite{LW19a}, a universal $\imath$quantum group $\tUi$ is defined as a coideal subalgebra of the Drinfeld double quantum group $\tU$, and the $\imath$quantum group $\Ui_\bvs$ with parameters can be recovered from $\tUi$ by a central reduction; see \cite{LW19a,WZ23}. We call $\Ui$ (or $\tUi$) quasi-split if $\I_\bullet =\emptyset$.

In the theory of quantum symmetric pairs, the quasi $K$-matrix was introduced in \cite[\S2.3]{BW18a} for $\Ui_\bvs$ with some constraints on parameters $\bvs$; a proof for greater generality was given in \cite{BK19}. In \cite{AV22, Ko21}, the authors removed the constraints on parameters and constructed the quasi $K$-matrix for general parameters; their results were reformulated for $\Ui_\bvs$ and adapted to $\tUi$ in \cite{WZ23}. The quasi $K$-matrix has been shown to be crucial in the construction of $\imath$canonical bases \cite{BW18a,BW21} and in the construction of relative braid group symmetries \cite{WZ23} for $\imath$quantum groups.
 
It was conjectured by Kolb and Pellegrini in \cite{KP11} that there exist relative braid group symmetries on the $\imath$quantum group $\Ui$ associated to the relative Weyl group of the underlying symmetric pair. For mostly quasi-split finite type, they {\em loc. cit.} constructed such symmetries with the help of computer calculations. Dobson formulated relative braid group symmetries for general type AIII and AIV in \cite{D19}. Recently, via a new approach using quasi $K$-matrices, Wang and the third author \cite{WZ23} constructed relative braid group symmetries on $\imath$quantum groups of arbitrary finite type, as well as compatible symmetries on finite-dimensional $\Ui$-modules. This construction has been generalized to the quasi-split Kac-Moody setting in \cite{Z23}. 

In \cite{LW19a}, the first author and Wang developed $\imath$Hall algebras of $\imath$quiver algebras to realize the universal quasi-split $\imath$quantum groups $\tUi$ (of finite type). The BGP type reflection functors are defined for $\imath$quiver algebras in \cite{LW21a}, which are used to obtain relative braid group symmetries $\TT_i$ on $\tUi$. For a class of quasi-split type ADE, the reflection functors can be used to define the $q$-root vectors in $\tUi$, and these $q$-root vectors lead to a natural PBW type basis for $\tUi$ and $\Ui$; see \cite{LW21a}. 

Since the relative braid group symmetries are now available for arbitrary finite type, it is a natural to construct PBW type bases for $\imath$quantum groups of arbitrary finite type.

On the other hand, 
Bao and Wang introduced in \cite{BW18b} the modified $\imath$quantum group $\dot{\U}^\imath$ and its integral $\A$-form $\dot{\U}^\imath_\A$, analogous to the construction of modified quantum groups in \cite{Lus93}.
In the same paper, they constructed an $\imath$-canonical basis for $\dot{\U}^\imath_\A$ with certain restrictions (see \cite[Section 3.4]{BW18b}) on the parameters $\bvs$; then many nice properties of $\dot{\U}^\imath_\A$, such as the freeness, follow from the existence of the $\imath$-canonical basis. It is natural to ask if these properties still hold for a larger set of parameters.


\subsection{Goal}
\label{subsec:goal}
The goal of this paper is to construct PBW type bases for $\imath$quantum groups $\tUi$ and $\Ui$ of arbitrary finite type and the integral form of modified $\imath$quantum groups $\dot{\U}^\imath$. The basic tools are the relative braid group symmetries and $\imath$-divided powers. 


\subsection{PBW bases for $\imath$quantum groups over $\Q(q)$}
\label{sec:intro2}
In the first part, we focus on constructing the PBW type bases for $\tUi$ over the field $\Q(q)$. The PBW type bases for $\Ui$ are obtained by a central reduction. 

 We briefly review the PBW bases for quantum groups. Let $W=\langle s_i\mid i\in\I\rangle$ be the Weyl group for $\g$ with the length function $\ell(\cdot)$ and the longest element $w_0$. Set $\cR^+$ to be the set of positive roots and $\cR^+(w)=\cR^+\cap w(\cR^-)$ for any $w\in W$. Let $\tTD_i$ be a variant of Lusztig's symmetries on $\tU$; see Proposition \ref{prop:braid1}. Using these braid group symmetries and a reduced expression of $w_0$, one can define root vectors $F_\beta\in\U^-(=\tU^-)$ for $\beta\in\cR^+$, and these root vectors lead to a PBW basis for $\U^-$; see Proposition~\ref{prop:PBW-QG}. 
 

Let $ (\I =\I_\circ \cup \I_\bullet, \tau)$ be a Satake diagram and $W^\circ=\langle \bs_i\mid i\in \wItau\rangle$ be the associated relative Weyl group. Set $W_\bullet=\langle s_i\mid i\in\I_\bullet\rangle$. The semidirect product $W_\bullet \rtimes W^\circ$ is identified with a subgroup of $W$, and the longest element $w_0$ of $W$ is identified with $\bbw_0w_\bullet$ in $W_\bullet \rtimes W^\circ$, where $\bbw_0 $ and $w_\bullet$ are the longest elements of $W^\circ$ and $W_\bullet$ respectively; cf. \cite{DK19}. We show in Lemma~\ref{lem:positive-roots-rankone} that $\cR^+$ is the disjoint union of two sets
\begin{align}\label{eq:intro1}
 \cR^+(\bbw_0)  \text{ and } \cR^+_\bullet,
\end{align}
where $\cR^+_\bullet$ is the positive system of $\I_\bullet$. 

Due to the natural decomposition \eqref{eq:intro1}, we will construct root vectors $B_\beta$ for $\beta\in \cR^+(\bbw_0)$ and $F_\gamma,E_\gamma$ for $\gamma\in \cR^+_\bullet$ separately. The monomials of all these root vectors for a fixed order will eventually provide the PBW type basis for $\tUi$.

By \cite[Theorem 4.2]{BW18b}, Lusztig's braid group symmetries $\widetilde{T}_i$ for $i\in\I_\bullet$ restrict to $\tUi$. 
Then $F_\gamma$, $E_\gamma$ can be constructed in the same way as the case of quantum groups using  $\widetilde{T}_i$ for $i\in\I_\bullet$. 

The nontrivial part is the construction of root vectors $B_{\beta}$ for $\beta\in  \cR^+(\bbw_0) $. In this case, Lusztig symmetries $\widetilde{T}_i$ for $i\in \wI$ do not restrict to $\tUi$, and it turns out that one should use the relative braid group symmetries $\TT_i$ introduced in \cite{WZ23} on $\tUi$.
Given a reduced expression $\bbw_0=\bs_{i_1}\bs_{i_2}\cdots\bs_{i_\ell}$ in $W^\circ$, there exists a unique $1\leq j\leq \ell$ and a unique $\beta_0\in\cR^+(\bs_{i_j})$ such that $\beta=\bs_{i_1}\bs_{i_2}\cdots \bs_{i_{j-1}}(\beta_0)$. One can regard $\cR^+(\bs_{i_j})$ as a subset of positive roots associated to the rank 1 subdiagram. Hence, we first construct root vectors $B_{\beta_0}$ for each rank 1 Satake diagram in Section \ref{sec:root-rank1}. 
In the higher rank case, we define the root vectors $B_\beta$ for $\beta\in\cR^+(\bbw_0)$ to be
\begin{align}\label{def:Bbetadiv}
    B_\beta:=\TT_{i_1}\TT_{i_2}\cdots \TT_{i_{j-1}}(B_{\beta_0});
\end{align}
see \eqref{def:Bbeta}. 
We define for any $\ba=(a_\beta)_{\beta\in\cR^+(\bbw_0)}\in\N^{\ell(\bbw_0)}$, $\bc=(c_{\gamma})_{\gamma\in\cR^+_\bullet}\in\N^{|\cR^+_\bullet|}$,
\begin{align*}
    B^{\ba}:=\prod_{\beta\in\cR^+(\bbw_0)} B_{\beta}^{a_\beta},
    \qquad
F_\bullet^{\bc}:=\prod_{\gamma\in\cR^+_\bullet}F_\gamma^{c_\gamma},\qquad E_\bullet^{\bc}:=\prod_{\gamma\in\cR^+_\bullet}E_\gamma^{c_\gamma},
\end{align*}
where orders of the products are induced by the given reduced expressions of $\bbw_0$ and $\bw$ in $W$. 

\begin{customthm}{\bf A}[Theorem \ref{thm:iPBW}]
Let $\bbw_0=\bs_{i_1}\bs_{i_2}\cdots\bs_{i_\ell}$ be a reduced expression of the longest element $\bbw_0$ in $W^\circ$, and $\bw=s_{j_1}s_{j_2}\cdots s_{j_r}$ be a reduced expression of the longest element $w_\bullet$ in $W_\bullet$. Then the monomials 
$B^\ba F_\bullet^{\bc} E_\bullet^{\bd}$ ($\ba\in \N^{\ell(\bbw_0)},\bc,\bd\in\N^{r}$) 
form a $\tU^{\imath0}$-basis of $\tUi$.
\end{customthm}

We explain the proof of this theorem. Following \cite[\S 2.4-2.5]{KY21}, the $\imath$quantum group $\tUi$ admits a filtration such that the associated graded algebra is isomorphic to $\tU^-\otimes\tU_\bullet^+\otimes \tU^{\imath0}$, where $\tU^{\imath0}$ is the Cartan part of $\tUi$ (see Section~\ref{sec:QSP} for precise definitions). We show that the leading term of $B_\beta$ equals $F_\beta$ for $\beta\in \cR^+(\bbw_0)$. In the rank 1 case, this property is clear from the explicit formulas of $B_\beta$; in the higher rank case, this property is proved in Lemma~\ref{lem:leadingPBW}, using the quasi $K$-matrix and the defining intertwining relation \eqref{eq:compTT} for $\TT_i$ established in \cite{WZ23}. This property implies that the images of monomials $B^\ba F_\bullet^{\bc} E_\bullet^{\bd}$ form a $\tU^{\imath0}$-basis for the associated graded algebra, and hence these monomials also form a basis for $\tUi$.

\subsection{Integral PBW bases for the modified $\imath$quantum group $\dot{\U}^\imath$}
\label{sec:intro3}

By \cite{BW18b,BW21}, the modified $\imath$quantum group contains the orthogonal idempotents $\mathbf{1}_\zeta$ ($\zeta\in X_\imath$), and is of the form 
$\dot{\U}^\imath=\bigoplus_{\zeta\in X_\imath}\dot{\U}^\imath\mathbf{1}_{\zeta}$,
where the $\imath$-weight lattice $X_\imath$ is defined in \eqref{def:XiYi}.

It was shown in \cite{WZ23} that the distinguished parameters $\bvs_\diamond$ (see Table \ref{table:localSatake}) are most natural for the construction of relative braid group symmetries. Since $q^{1/2}$ appears in the distinguished parameters, we shall consider the ring $\bar{\A}=\Z[q^{1/2},q^{-1/2}]$ and an integral $\bar{\A}$-form $\dot{\U}^\imath_{\bar{\A}}$ of $\dot{\U}^\imath=\dot{\U}^\imath_\bvs$ defined in a similar way to \cite[Definition 3.19]{BW18b}; see Definition~\ref{def:integral-Uidot}.

In order to construct a PBW type basis for $\dot{\U}^\imath_{\bar\A}$, we need to establish ``$\imath$-divided powers'' $B_{\beta}^{(m)},\beta\in\cR^+(\bbw_0),m\geq0$ (which differ from usual divided powers in general) for root vectors $B_\beta$. 
For each rank 1 Satake diagram, we construct $B_{\beta}^{(m)}$ case by case in Section~\ref{subsec:divid-root-rank1}, using $\imath$-divided powers $B_{i}^{(m)}$ (=$B_{\alpha_i}^{(m)}$) given in \cite{BW18b,BW21}. We prove that $B_{\beta}^{(m)}\mathbf{1}_\zeta$ lie in the integral form $\dot{\U}^\imath_{\bar{\A}}$ if the parameters $\vs_i\in\bar{\A}$ for $i\in\I_\circ$; see Lemma \ref{lem:div-integral}. In the higher rank case, we define $B_\beta^{(m)}$ similar to \eqref{def:Bbetadiv} using the relative braid group symmetries:
\begin{align}\label{def:Bbetadiv-2}
    B_\beta^{(m)}:=\bT_{i_1}\bT_{i_2}\cdots \bT_{i_{j-1}}(B_{\beta_0}^{(m)});
\end{align}
see \eqref{def:Bbeta-div}. The divided powers $E_\gamma^{(m)}$, $F_{\gamma}^{(m)}$ for $\gamma\in\cR^+_\bullet$ are the same as Lusztig's.

It remains to show that 
$B_\beta^{(m)}\mathbf{1}_\zeta\in \dot{\U}^\imath_{\bar{\A}}$ in the higher rank case. The relative braid group symmetry $\bT_i$ on $\Ui$ induces an automorphism, also denoted by $\bT_{i}$, on $\dot{\U}^\imath$. When parameters satisfy $\vs_i =q^{a_i}\vs_{\diamond,i},a_i\in\Z,i\in\I_\circ$, we show that $\dot{\U}^\imath_{\bar{\A}}$ is invariant under the actions of $\bT_i$ in Theorem \ref{prop:integral-braid}. This result is first established for the distinguished parameters $\bvs_\diamond$ using the integrality of quasi $K$-matrices in Proposition~\ref{lem:integral-braid}, and then extended to the general case using the isomorphism $\phi_{\bvs,\bvs'}$ (see Lemma~\ref{lem:div-iso}) for modified $\imath$quantum groups associate to different parameters. Theorem \ref{prop:integral-braid} implies the integrality of $B_\beta^{(m)}\mathbf{1}_\zeta$ by \eqref{def:Bbetadiv-2}.

Now we define 
\begin{align*}
    B^{(\ba)}:=\prod_{\beta\in\cR^+(\bbw_0)} B_{\beta}^{(a_\beta)}, \qquad F_\bullet^{(\bc)}:=\prod_{\gamma\in\cR^+_\bullet}F_\gamma^{(c_\gamma)},\qquad E_\bullet^{(\bc)}:=\prod_{\gamma\in\cR^+_\bullet}E_\gamma^{(c_\gamma)},
\end{align*}
where orders of the products is induced by a given reduced expression of $w_\bullet$ and $\bbw_0$.

\begin{customthm}{\bf B}[Theorem \ref{thm:PBWUidot}]
Let $\bvs=(\vs_i)_{i\in\I_\circ}$ be parameters such that $\vs_i$ has the form $ \vs_i =q^{a_i}\vs_{\diamond,i},a_i\in\Z$ for any $i\in\I_\circ$.
Let $\bbw_0=\bs_{i_1}\bs_{i_2}\cdots\bs_{i_\ell}$ be a reduced expression of the longest element $\bbw_0$ in $W^\circ$, and $w_\bullet=s_{j_1}s_{j_2}\cdots s_{j_r}$ a reduced expression of the longest element $w_\bullet$ in $W_\bullet$. Then the monomials 
$B^{(\ba)} F_\bullet^{(\bc)} E_\bullet^{(\bd)}\mathbf{1}_\zeta$ ($\ba\in \N^{\ell(\bbw_0)},\bc,\bd\in\N^{|\cR_\bullet^+|}$, $\zeta\in X_\imath$)  
form an $\bar{\A}$-basis of $\dot{\U}^\imath_{\bar{\A}}$.    
\end{customthm}

It is worth noting that the parameters $\bvs$ given in Theorem {\bf B} are different from those considered in \cite{BW18b,BW21} to construct $\imath$-canonical bases. Thus our results imply that the integral form $\dot{\U}^\imath_{\bar{\A}}$ is a free $\bar{\A}$-module for a new class of parameters; we further show that $\dot{\U}^\imath_{\bar{\A}}$ is linearly isomorphic to the integral form of the parabolic subalgebra; see Proposition \ref{prop:UiPintegral}. 

Unlike the case of $\U^-$, there is no known PBW basis approach for the canonical basis on the modified quantum groups. Hence, one may not expect a construction of the $\imath$canonical basis on $\dot{\U}^\imath$ using the PBW basis.

\subsection{Organization}
The paper is organized as follows. We review quantum groups, $\imath$quantum groups and their braid group symmetries  in section \ref{sec:prelim}. In section \ref{sec:root-rank1}, we construct the $q$-root vectors for rank 1 Satake diagrams. Section \ref{sec:PBW-Ui} is devoted to formulating the PBW bases of $\imath$quantum groups.
In Section \ref{sec:Uidot}, we studied general properties of modified $\imath$quantum groups and define the $\imath$divided powers of root vectors for $\imath$quantum groups. In Section \ref{subsec:braid-Uidot}, we construct the PBW bases for the integral form of modified $\imath$quantum groups.

\vspace{2mm}
\noindent{\bf Acknowledgments.} 
We thank Weiqiang Wang for very helpful discussions. WZ thanks Huanchen Bao and Jinfeng Song for inspiring discussions in the early stage of this project. ML thanks the Institute of Mathematics at Academia Sinica (Chinese Taipei) for hospitality
and support which has greatly facilitated the completion of this work. ML is partially supported by the National Natural Science Foundation of China (No. 12171333, 1226113149). WZ is partially supported by  the New Cornerstone Foundation through the
New Cornerstone Investigator grant, and by Hong Kong RGC grant 14300021, both awarded to Xuhua He.
      
\vspace{2mm}

\noindent{\bf Statements and Declarations} 
The authors have no competing interests to declare that are relevant to the content of this article.
The authors declare that the data supporting the findings of this study are available within the paper.
\section{Preliminaries}
\label{sec:prelim}

\subsection{Quantum groups}


Let $(\I,\cdot)$ be a Cartan datum of finite type and $(Y,X,\langle \cdot ,\cdot\rangle,\dots)$ be a root datum of type $(\I,\cdot)$; cf. \cite[2.2.1]{Lus93}. Then we have
\begin{itemize}
\item[(a)] two finitely generated free abelian groups $Y,X$ and a perfect bilinear pairing $\langle\cdot,\cdot\rangle:Y\times X\rightarrow\Z$;
\item[(b)] an embedding $\I\subset X$ ($i\mapsto \alpha_i$) and an embedding $\I\subset Y$ ($i\mapsto h_i$) such that $\langle h_i,\alpha_j\rangle =2\frac{i\cdot j}{i\cdot i}$ for all $i,j\in \I$.
\end{itemize}

We always assume that the root datum $(Y,X,\langle \cdot ,\cdot\rangle,\dots)$ is {\em simply connected}. i.e., $Y=\bigoplus_{i\in\I} \Z h_i$, and $X=\Hom(Y,\Z)$.  
We call $X$ the \emph{weight lattice} and $Y$ the \emph{coroot lattice}.
  
Let $C:= (c_{ij}):=(\langle h_i, \alpha_j\rangle)$ be the Cartan matrix and $D$ be the diagonal matrix
\[
D=\diag(\epsilon_i\mid i\in \I),  \text{ where } \epsilon_i=\frac{i\cdot i}{2}  \;\; (\forall i\in \I).
\]
 Then $DC$ is symmetric.

Let $s_i:Y\rightarrow Y$ be the homomorphism given by $s_i(\mu)=\mu-\langle \mu,\alpha_i\rangle h_i$.  
Let $W=\langle s_i\mid i\in \I \rangle$ be the Weyl group. We denote by $\ell(\cdot)$ the length function on $W$. Let $\cR=\{w\alpha_i\in X\mid w\in W, i\in\I\}$ be the set of roots, and denote by $\cR^+$ the corresponding set of positive roots. Similarly, denote by $\cR^\vee\subseteq Y$  the set of coroots. We denote by $\rho^\vee\in Y$ the half sum of all positive coroots, and by $\rho\in X$ the half sum of all positive roots.

Let $q$ be an indeterminate and $\Q(q)$ be the field of rational functions in $q$ with coefficients in $\Q$, the field of rational numbers. 
We denote
\[
q_i:=q^{\epsilon_i}, \qquad \forall i\in \I.
\]
Denote, for $r,m \in \N,i\in \I$,
\[
 [r]_i =\frac{q_i^r-q_i^{-r}}{q_i-q_i^{-1}},
 \quad
 [r]_i!=\prod_{s=1}^r [s]_i, \quad \qbinom{m}{r}_i =\frac{[m]_i [m-1]_i \ldots [m-r+1]_i}{[r]_i!}.
\]

For $A, B$ in a $\Q(q)$-algebra, we shall denote $[A,B]_{q^a} =AB -q^aBA$, and $[A,B] =AB - BA$. The quantum group $\U$ is defined to be the $\Q(q)$-algebra generated by $E_i,F_i$ for $i\in\I$, and $\tK_\mu,$, $\mu\in Y$, subject to the following relations:
\begin{align}
&K_0=1,\; K_{\mu}K_{\nu}=K_{\mu+\nu}, \quad \quad \forall \mu,\nu\in Y;\label{eq:KK}
\\
&K_{\mu}E_i=q^{\langle \mu,\alpha_i\rangle} E_iK_\mu, \quad K_\mu F_i=q^{-\langle \mu,\alpha_i\rangle}F_iK_\mu, \quad \forall i\in \I, \mu\in Y.\\
&[E_i,F_j]=\delta_{ij}\frac{\tK_i-\tK_i^{-1}}{q_i-q_i^{-1}}, \text{ where }\tK_i=K_{h_i}^{\epsilon_i}, \; \;\forall i\in \I.\label{Q4}
\end{align}
and the quantum Serre relations, for $i\neq j \in \I$,
\begin{align}
& \sum_{r=0}^{1-c_{ij}} (-1)^r  E_i^{(r)} E_j  E_i^{(1-c_{ij}-r)}=0,
  \label{eq:serre1} \\
& \sum_{r=0}^{1-c_{ij}} (-1)^r  F_i^{(r)} F_j  F_i^{(1-c_{ij}-r)}=0,
  \label{eq:serre2}
\end{align}
where $E_i^{(r)}=\frac{E_i^r}{[r]_i^!}$, $F_i^{(r)}=\frac{F_i^r}{[r]_i^!}$ are the divided powers.




Let $\U^+, \U^0, \U^-$ be the subalgebras of $\U$ generated by $\{E_i\mid i\in \I\}, \{K_\mu\mid \mu\in Y\}, \{F_i\mid i\in \I\}$, respectively. There is a natural triangular decomposition $\U\cong \U^+\otimes \U^0\otimes \U^-$. Then $\U$ is a $\Z\I$-graded algebra $\U=\oplus_{\mu\in\Z\I}\U_\mu$ by setting $\deg E_i=-\alpha_i$, $\deg F_i=\alpha_i$, and $\deg K_\mu=0$. Note that $\U_0=\U^0$. 
 Moreover, $\U^+$ is a $\Z\I$-graded algebra by $\deg E_i=-\alpha_i$, we denote the grading algebra by $\U^+=\oplus_{\mu} \U^+_\mu$.
The comultiplication $\Delta: \U \rightarrow \U \otimes \U$ is given by
\begin{align}  \label{eq:Delta}
	\begin{split}
		\Delta(E_i)  = E_i \otimes 1 + \tK_i \otimes E_i, & \quad \Delta(F_i) = 1 \otimes F_i + F_i \otimes \tK_{i}^{-1}, \\
		\Delta(\tK_\mu) = \tK_\mu\otimes \tK_\mu, &\quad \forall i\in\I, \mu\in Y.
	\end{split}
\end{align}

We also have a Drinfeld double quantum group $\tU$ (see e.g. \cite{LW19a}), which is  is defined to be the $\Q(q)$-algebra generated by $E_i,F_i, \tK_i,\tK_i'$, $i\in \I$, where $\tK_i, \tK_i'$ are invertible, subject to the following relations:
\begin{align}
[E_i,F_j]= \delta_{ij} \frac{\tK_i-\tK_i'}{q_i-q_i^{-1}},  &\qquad [\tK_i,\tK_j]=[\tK_i,\tK_j']  =[\tK_i',\tK_j']=0,
\label{eq:KK}
\\
\tK_i E_j=q_i^{c_{ij}} E_j \tK_i, & \qquad \tK_i F_j=q_i^{-c_{ij}} F_j \tK_i,
\label{eq:EK}
\\
\tK_i' E_j=q_i^{-c_{ij}} E_j \tK_i', & \qquad \tK_i' F_j=q_i^{c_{ij}} F_j \tK_i',
 \label{eq:K2}
\end{align}
and \eqref{eq:serre1}--\eqref{eq:serre2}.
Similar to $\U$, we can define subalgebras $\tU^+\cong \U^+,\tU^-\cong\U^-,\tU^0$. 
The comultiplication $\Delta:\tU\rightarrow \tU\otimes\tU$ is given by
\begin{align}  
	\begin{split}
		\Delta(E_i)  = E_i \otimes 1 + \tK_i \otimes E_i, & \quad \Delta(F_i) = 1 \otimes F_i + F_i \otimes \tK_{i}', \\
		\Delta(\tK_i) = \tK_i\otimes \tK_i, &\quad\Delta(\tK'_i) = \tK'_i\otimes \tK'_i,\quad \forall i\in\I.
	\end{split}
\end{align}

There exists an anti-involution $\sigma$ on $\tU$ that fixes $E_i,F_i$ and swap $K_i\leftrightarrow K_i'$, for $i\in\I$. 

Note that $\tK_i \tK_i'$ are central in $\tU$ for all $i\in \I$. We have a canonical homomorphism $$\pi: \tU\rightarrow \U$$ 
given by $E_i\mapsto E_i$, $F_i\mapsto F_i$, $K_i\mapsto K_i$, $K_i'\mapsto K_i^{-1}$, for $i\in\I$. Obviously, $\ker\pi=(K_iK_i'-1\mid i\in\I)$.

For $i\in\I$, denote by $r_{i}: \U^+ \rightarrow \U^+$ the unique $\Q(q)$-linear maps \cite{Lus93}  such that
\begin{align}  \label{eq:rr}
r_{i}(1) = 0, \quad r_{i}(E_{j}) = \delta_{ij},
\quad r_{i}(xx') = xr_{i}(x') + q_i^{-\langle h_i,\mu'\rangle} r_{i}(x)x',
\end{align}
for $x \in \U^+_{\mu}$ and $x' \in \U^+_{\mu'}$; cf. $\deg E_i=-\alpha_i$, $i\in\I$. Since $\tU^+$ is naturally isomorphic to $ \U^+$, we also have a linear map $r_i$ on $\tU^+$.


\subsection{Braid group symmetries and PBW bases of quantum groups}

Let $\Br(W)$ be the braid group associated to the Weyl group $W$, generated by $t_i$ ($i\in\I$).
Lusztig introduced braid group symmetries $T_{i,e}',T_{i,e}''$ for $i\in\I$ and $e=\pm1$, on the Drinfeld-Jimbo quantum group $\U$ \cite[\S37.1.3]{Lus93}. Denote $T_i:=T_{i,1}''$. These braid group symmetries can be lifted to  the Drinfeld double $\tU$; see \cite[Propositions 6.20–6.21]{LW21}, which are  denoted by $\widetilde{T}_{i,e}',\widetilde{T}_{i,e}''$. We recall $\widetilde{T}_i:=\widetilde{T}_{i,1}''$ in the following.  In fact, by using the anti-involution $\sigma$ of $\tU$, we have 
\begin{align}
\widetilde{T}_i^{-1}=\sigma\circ\widetilde{T}_i\circ\sigma,\qquad \forall i\in\I.
\end{align}


\begin{proposition}
[\cite{Lus90a}]
   \label{prop:braid1}
Set $r=-c_{ij}$. There exists an automorphism $\widetilde{T}_{i}$, for $i\in \I$, on $\tU$ such that
\begin{align*}
&\widetilde{T}_{i}(K_\mu)= K_{s_i(\mu)},
\qquad \widetilde{T}_{i}(K'_\mu)= K'_{s_i(\mu)},\;\;\forall \mu\in \Z\I\\
&\widetilde{T}_{i}(E_i)=- F_i  (K_i') ^{-1},\qquad \widetilde{T}_{i}(F_i)=- K_i^{-1}E_i,\\
&\widetilde{T}_{i}(E_j)= \sum_{s=0}^r (-1)^s q_i^{-s} E_i^{(r-s)} E_j E_i^{(s)},\qquad  j\neq i,\\
&\widetilde{T}_{i}(F_j)= \sum_{s=0}^r (-1)^s q_i^{s} F_i^{(s)} F_j F_i^{(r-s)}, \qquad  j\neq i.
\end{align*}
Moreover, there exists a group homomorphism $\Br(W)\rightarrow \Aut(\tU)$, $t_i\mapsto \widetilde{T}_i$ for $i\in\I$.    
\end{proposition}

Hence, we can define
\begin{align}\widetilde{T}_w 
 := \widetilde{T}_{i_1}\cdots
\widetilde{T}_{i_r} \in \Aut(\tU),
\end{align}
where $w = s_{i_1}
\cdots s_{i_r}$
is any reduced expression of $w \in W$. 

Set $\A=\Z[q,q^{-1}]$. 
Let $\tU_{\A}^-$ be the $\A$-subalgebra of $\tU^-$ generated by $ F_i^{(a)}$ for  $i\in \I,a\in \N$. Similarly, one can define $\tU_{\A}^+$.

Let $w\in W$ with a reduced expression ${w}=s_{i_1} \cdots s_{i_l}$. We set
\begin{align}
\cR^+(w)=\{\alpha_{i_1}, s_{i_1}(\alpha_{i_2}), s_{i_1}s_{i_2}(\alpha_{i_3}), \ldots, s_{i_1}s_{i_2}\cdots s_{i_{l-1}}(\alpha_{i_l})\}=\cR^+\cap w(\cR^-).
\end{align}
Let $w_0$ be the longest element of $W$. Then $\cR^+(w_0)=\cR^+$.

\begin{proposition}[\text{\cite[Proposition 1.10]{Lus90a},\cite[Proposition 5.7]{Lus90c}}]
\label{prop:PBW-QG}
Fix a reduced expression ${w_0}=s_{i_1} \cdots s_{i_l}$ for the longest element $w_0\in W$. Set
\begin{align}
\label{eq:F-root}
F_{\beta_k}=\widetilde{T}_{i_1}\widetilde{T}_{i_2}\cdots \widetilde{T}_{i_{k-1}}(F_{i_k}),\qquad 
F_{\beta_k}^{(a)}=\widetilde{T}_{i_1}\widetilde{T}_{i_2}\cdots \widetilde{T}_{i_{k-1}}(F_{i_k}^{(a)}),
\end{align}
for $1\leq k\leq l,a\in \N$. Then the monomials 
$$F^{\ba}=F_{\beta_1}^{a_1} F_{\beta_2}^{a_2}\cdots F_{\beta_l}^{a_l}, \qquad \ba=(a_1,\ldots, a_l)\in \N^l$$ 
form a $\Q(q)$-basis for $\tU^-$, and  the monomials $$F^{(\ba)}:=F_{\beta_1}^{(a_1)} F_{\beta_2}^{(a_2)}\cdots F_{\beta_l}^{(a_l)}, \qquad \ba=(a_1,\ldots, a_l)\in \N^l$$ form an $\A$-basis for $\tU_\A^-$.
\end{proposition}

In fact, for any total order on $\cR^+$, the monomials 
$$F^\ba=\prod_{\beta\in\cR^+} F_\beta^{a_\beta},\qquad \ba=(a_\beta)\in\N^l$$
form a $\Q(q)$-basis for $\tU^-$; see \cite[Proposition 1.10]{Lus90a}.

\subsection{Satake diagrams and relative root systems}


Given a subset $\bI\subset \I$, denote by $W_{\bullet}$ the parabolic subgroup of $W$ generated by $s_i,i\in \bI.$ Set $\bw$ to be the longest element of $W_\bullet$. Let $\cR_{\bullet}$ be the set of roots with the simple system $\{\alpha_i\mid i\in \bI\}$.  Similarly, $\cR_{\bullet}^\vee$ is the set of coroots which lie in the span of $h_i,i\in \bI.$
Let $\rho_\bullet$ be the half sum of positive roots in the root system $\cR_{\bullet}$, and $\rho_\bullet^\vee$ be the half sum of positive coroots in $\cR_\bullet^\vee$.

An {\em admissible pair} $(\I=\bI \cup \wI,\tau)$ (cf. \cite{Ko14}) consists of a partition $\bI\cup \wI$ of $\I$, and an involution $\tau$ of the Cartan datum $(\I,\cdot)$ (possibly $\tau=\Id$) such that
\begin{itemize}
\item[(1)] $\tau(\bI)=\bI$, 
\item[(2)]
 $\bw(\alpha_j) = - \alpha_{\tau j}$ for $j\in \bI$,
\item[(3)]
 If $j\in \wI$ and $\tau j =j$, then $\langle \rho_{\bullet}^\vee,\alpha_j\rangle\in \Z$.
\end{itemize}

The partition $\bI\cup \wI$ can be visualized as a bicoloring on the Dynkin diagram, where vertices in $\bI$ are colored black and vertices in $\wI$ are colored white; $\tau$ is identified with a Dynkin diagram involution. This bicolored Dynkin diagram decorated by the diagram involution $\tau$ is known as the Satake diagram. We shall interchangeably use the terminologies: admissible pairs and Satake diagrams.

A symmetric pair $(\g,\theta)$ (of finite type) consists of a semisimple Lie algebra $\g$ and an involution $\theta$ on $\g$. For each Satake diagram, there is an associated involution $\theta$ on $\g$, whose action on the weight lattice is given by $\theta=-\bw \circ \tau$. The irreducible symmetric pairs are classified by Satake diagrams.

Let $(Y,X,\langle \cdot ,\cdot\rangle,\dots)$ be a root datum of type $(\I,\cdot)$. Then  $\tau$ extends to an involution on $X$ and an involution on $Y$, such that the perfect bilinear pairing is invariant under the involution $\tau$. Furthermore,  $\theta=-w_\bullet\circ \tau$ is an involution of $X$ and $Y$. Following \cite{BW18b}, the $\imath$-weight lattice and $\imath$-root lattice are defined to be:
\begin{align}\label{def:XiYi}
X_\imath=X/\check{X}, \text{ where }\check{X}=\{\lambda-\theta(\lambda)\mid \lambda\in X\},
\quad
Y^\imath=\{\mu\in Y\mid \theta(\mu)=\mu\}.
\end{align}
Denote by $\ov{\lambda}$ the image of  $\lambda\in X$ in $X_\imath$. 

Let $\wItau$ be a fixed set of representatives for $\tau$-orbits on $\wI$. The {\em real rank} of a Satake diagram $(\I=\bI \cup \wI,\tau)$ is defined to be the cardinality of $\wItau$.

We define real rank 1 Satake subdiagrams following the convention in \cite{DK19,WZ23}. For $i\in \wI$, we set $\I_{\bullet,i}:=\bI\cup \{i,\tau i\}$ and there is a real rank 1 Satake subdiagram $(\I_{\bullet,i}=\{i,\tau i\}\cup\bI,\tau\big|_{\I_{\bullet,i}})$.
Let $\cR_{\bullet,i}$ be the set of roots which are linear combinations of $\alpha_j,j\in \I_{\bullet,i}$; $\cR_{\bullet,i}$ is naturally a root system with the simple system $\{\alpha_j\mid j\in \I_{\bullet,i}\}$. Denote by $\cR_{\bullet,i}^+$ the corresponding positive system of $\cR_{\bullet,i}$. Let $W_{\bullet,i}$ be the parabolic subgroup of $W$ generated by $s_i,i\in \I_{\bullet,i}$ and $w_{\bullet,i}$ the longest element of $W_{\bullet,i}$.

Define $\bs_i\in W_{\bullet,i}$ such that
\begin{align}\label{def:bsi}
\bwi= \bs_i w_\bullet \, (=w_\bullet \bs_i).
\end{align}
It follows from the definition of admissible pairs that $\bwi,\bs_i,$ and $w_\bullet$ commute with each other and
$\ell(\bwi) = \ell(\bs_i) + \ell(w_\bullet)$. Then the subgroup of $W$, known as the relative Weyl group,
\[
\reW :=\langle \bs_i\mid i\in \wItau \rangle,
\]
is a Weyl group by itself with its simple reflections identified with $\{\bs_i \mid i\in \wItau\}$.

Let $\bbw_0$ be the longest element in $W^\circ$. It is known that $w_0=\bbw_0 \bw=\bw \bbw_0 $ and $\ell(w_0)=\ell(\bbw_0)+\ell(\bw)$.

\begin{lemma}
\label{lem:positive-roots-rankone}
\begin{itemize}
\item[(1)] $\cR^+$ is the disjoint union of $\cR_{\bullet}^+$ and $\cR^+(\bbw_0)$.
\item[(2)] $\cR_{\bullet,i}^+$ is the disjoint union of $\cR_{\bullet}^+$ and $\cR^+(\bs_i)$.
\end{itemize}
\end{lemma}

\begin{proof}
Statement (2) is obtain by applying (1) to the real rank 1 subdiagram $(\I_{\bullet,i}=\{i,\tau i\}\cup\bI,\tau\big|_{\I_{\bullet,i}})$. It suffices to prove (1). 

Since $w_0 =\bbw_0\bw $ and $\ell(w_0)=\ell(\bbw_0)+\ell(\bw)$, a reduced expression of $w_0$ can be obtained by composing reduced expressions of $\bw$ and $\bbw_0$ in $W$. Then we have
\begin{align}
\cR^+=\cR^+(w_0)=\cR^+(\bbw_0) \sqcup \bbw_0\big( \cR^+(\bw)\big).
\end{align}
By the definition \eqref{def:bsi}, $\bs_i$ permutes $\cR^+(\bw)$ and hence $\bbw_0$ permutes $\cR^+(\bw)$. Therefore, we have $ \bbw_0\big( \cR^+(\bw)\big)= \cR^+(\bw)$ and then (1) follows.
\end{proof}

We introduce a subgroup of $W$:
\begin{align}
W^\theta = \{w \in W \mid w\theta = \theta w\}.
\end{align}
By definition, $\bs_i,s_j\in W^\theta $ for $i\in \wItau, j\in \bI$. Then $W_\bullet $ and $ W^\circ$ can be naturally identified as subgroups of $ W^\theta$. Hence, $W_\bullet $ and $ W^\circ$ acts naturally on $X_\imath$ and $Y^\imath$.

It is well known that (see, e.g., \cite[\S2.2]{DK19})
\begin{align}\label{eq:Wproduct}
W^\theta= W_\bullet \rtimes W^\circ .
\end{align}
We shall refer to the braid group associated to the relative Weyl group $W^\circ$ as the
relative braid group and denote it by $\Br(W^\circ
)$. Accordingly, we denote the braid group
associated to $W_\bullet$ by $\Br(W_\bullet)$.

\subsection{Quantum symmetric pairs and $\imath$quantum groups}
\label{sec:QSP}

We recall the definition of quantum symmetric pairs from \cite{Let02, Ko14}; see also \cite{LW19a,WZ23}.  Given a Satake diagram $(\I=\wI\cup\bI,\tau )$, the {\em universal  $\imath$quantum group} $\tUi$ is the subalgebra of $\tU$ generated by
\begin{align}
\label{def:iQG}
\begin{split}
&B_i=F_i+ \tTD_{\bw}( E_{\tau i}) K_i', \qquad k_i^{\pm1}=(K_i K_{\tau i}')^{\pm1} \qquad (i\in \wI),
\\
&E_j,\quad B_j=F_j,\quad K_j^{\pm 1},\quad (K_j')^{\pm1}\qquad (j\in \bI).
\end{split}
\end{align}
By definition, $\tUi$ contains the Drinfeld double $$\tU_\bullet=\langle E_j,F_j,K_j^{\pm 1},(K_j')^{\pm1}\mid j\in \bI\rangle$$ 
associated to
$\I_\bullet$ as a subalgebra.
We define $\tbU^\pm,\tbU^0$ in a similar way as $\tU^\pm, \tU^0$ and then there is a natural triangular decomposition $\tbU\cong \tbU^+ \otimes \tbU^0 \otimes \tbU^-$.
Denote $\tU^{\imath 0}$  the subalgebra of $\tUi$ generated by $k_i^{\pm1},K_j^{\pm 1},(K_j')^{\pm 1},\;i\in \wI,j\in \bI$.

The pair $(\tU,\tUi)$ is known as the (universal) quantum symmetric pair associated to $(\I=\wI \cup \bI,\tau )$. 

By \cite[Proposition 3.12]{WZ23}, there exists a unique anti-involution $\sigma^\imath$ of $\tUi$ such that 
\begin{align}
    \sigma^\imath(B_i)=B_i,\qquad \sigma^\imath(x)=\sigma(x),\text{ for }i\in\I_\circ,x\in\tU^{\imath0}\tU_\bullet.
\end{align}


 Set $\deg B_i=1$ ($i\in\I_\circ$), $\deg F_i= 0=\deg E_i$ ($i\in\bI$) and $\deg \tk_i=0$ ( $i\in\I_\circ$), $\deg K_i=0=\deg K_i'$ ($i\in\I_\bullet$). 
 By \cite[(2.21)]{KY21}, we have the following algebra isomorphism 
 \begin{align}
 \label{eq:grade}
 \phi: \tU^- \otimes \tU_{\bullet}^+ \otimes \tU^{\imath0}\longrightarrow \tU^{\imath,\gr},\qquad
 F_i\mapsto \begin{cases} B_i& \text{ for } i\in\I_\circ, \\ F_i&\text{ for }i\in\I_\bullet.\end{cases}
 \end{align}

Let $\bF$ be the algebraic closure of $\Q(q)$. 
Consider the quantum group $\U$ over the field $\bF$, which is denoted by $\U_\bF$. 
The Letzter's $\imath$quantum group $\Ui=\Ui_{\bvs}$, with parameters 
$\bvs=(\vs_i)_{i\in\I_\circ}\in (\bF^\times)^{\I_\circ}$, 
is the $\bF$-subalgebra of the quantum group $\U_\bF$  generated by
\begin{align*}
B_i=F_i+\vs_i T_{w_\bullet}(E_{\tau i})\tK_i^{-1}\; (i\in \I_\circ), \quad K_\mu \;(\mu\in Y^\imath),
\quad
F_i \; (i\in \I_\bullet), \quad E_i\;(i\in\I_\bullet).
\end{align*}
By definition, the algebra $\Ui$ contains $\U_\bullet$ (over $\bF$) as a subalgebra.  When the parameters $\bvs\in(\Q(q)^\times)^{\I_\circ}$, we can also consider $\Ui$ to be over the field $\Q(q)$. So we sometimes denote the $\bF$-algebra $\Ui$ by $\Ui_\bF$ to emphasis the ground field $\bF$.

Consider $\tUi$ over the field $\bF$, which is denoted by $\tUi_\bF$. 
There is a canonical homomorphism 
\begin{align}\label{def:pibvs}
\pi^\imath_\bvs: \tUi_\bF\rightarrow \Ui_\bvs
\end{align}
given by
$B_i\mapsto B_i$, 
$k_i\mapsto \vs_{\tau i} K_iK_{\tau i}^{-1}$ ($i\in\I_\circ$), $K_i\mapsto K_i$ ($i\in\I_\bullet$), $K_i'\mapsto K_i^{-1}$ ($i\in\I_\bullet$); see \cite[Proposition 2.8]{WZ23},  \cite[Proposition 6.2]{LW19a}.





\subsection{Relative braid group symmetries}
In order to give the braid group symmetries on $\imath$quantum groups, we shall consider quantum groups and $\imath$quantum groups over the algebraic closure $\bF$ of $\Q(q)$, which are denoted by $\tU_\bF,\tUi_\bF$ respectively.

\begin{proposition}[\text{\cite[Theorem~3.6]{WZ23}}]
\label{prop:Kmatrix}
There exists a unique element $\widetilde{\Upsilon}=\sum_{\mu\in\N\I}\widetilde{\Upsilon}^\mu$ such that $\widetilde{\Upsilon}^0=1$, $\widetilde{\Upsilon}^\mu\in\tU^+_\mu$ and the following identities hold:
\begin{align}
B_i\widetilde{\Upsilon}=&\widetilde{\Upsilon}B_i^\sigma,\qquad (i\in\I_\circ), 
\\
x\widetilde{\Upsilon}=&\widetilde{\Upsilon}x,\qquad (x\in\tU^{\imath0}\tU_\bullet),
\end{align}
where $B_i^\sigma:=\sigma(B_i)=F_i+K_i\widetilde{T}_{w_\bullet}^{-1}(E_{\tau i})$.  Moreover, $\widetilde{\Upsilon}^\mu=0$ unless $\theta(\mu)=-\mu$.
\end{proposition}

Let $\ba = (a_i)_{i\in\I} \in (\bF^\times)^\I
$. By \cite[Proposition 2.1]{WZ23}, there exists an
automorphism $\widetilde{\Psi}_\ba$ on the $\bF$-algebra $\tU_\bF$ such that
\begin{align}
\widetilde{\Psi}_\ba:\tK_i\mapsto a_i^{1/2}\tK_i, \quad \tK_i'\mapsto a^{1/2}_i\tK_i',\quad E_i\mapsto a_i^{1/2}E_i,\quad F_i\mapsto F_i.
\end{align}
We have an automorphism $\Phi_\ba$ on the $\bF$-algebra $\U$ such that
\begin{align}
\label{def:Phi}
    \Phi_\ba:K_i\mapsto K_i,\quad E_i\mapsto a_i^{1/2}E_i,\quad F_i\mapsto a_i^{-1/2}F_i.
\end{align}

Let $\bvs_\diamond=(\vs_{\diamond,i})_{i\in\I}$ be the scalars such that (cf. \cite[(2.21)]{WZ23})
\begin{align}
\vs_{\diamond,i}=\begin{cases}
    -q_i^{-\langle h_i,\alpha_i+w_\bullet\alpha_{\tau i}\rangle/2},&\text{ if } i\in\I_\circ,
    \\
    1,&\text{ if }i\in\I_\bullet.
\end{cases}
\end{align}

Define the following rescaled version $\tT_i$ of Lusztig's braid group symmetries for $i\in\I$ by
$
\tT_i:=\widetilde{\Psi}_{\bvs_\diamond}^{-1}\circ \widetilde{T}_{i}\circ\widetilde{\Psi}_{\bvs_\diamond}.
$

For $i\in\I_{\circ}$, denote by $\widetilde{\Upsilon}_i$ the quasi K-matrix associated to the real rank 1 subdiagram 
$(\I_{\bullet,i}=\{i,\tau i\}\cup\I_\bullet,\tau|_{\I_{\bullet,\tau}})$. 

\begin{proposition}
\label{prop:braid-iQG}
(a) \cite[Theorem 4.7]{WZ23} For $i\in\I_{\circ,\tau}$, there exists an automorphism $\TT_i$ on $\tUi$ such that
\begin{align}
\label{eq:compTT}
\TT_i^{-1}(x) \fX_i =\fX_i \tT_{\bs_i}^{-1}(x),\qquad \forall x\in \Ui.
\end{align}


(b) \cite[Theorem 4.2]{BW18b} For $j\in\I_\bullet$, the automorphism $\tT_j=\widetilde{T}_j$ on $\tU_\bF$ restricts to an automorphism of $\tUi$. Moreover, the action of $\tT_j$ on $B_i$ $ (i\in\I_\circ)$ is given by
\begin{align}
\tT_j(B_i)=\sum_{s=0}^r(-1)^sq_j^s F_j^{(s)}B_iF_j^{(r-s)}, \qquad \text{ for }r=-c_{ij}.
\end{align}    

(c) \cite[Theorem 9.3]{WZ23} There exists a braid group action of $\Br(W^\theta)=\Br(W_\bullet) \rtimes \Br(W^\circ
)$ on $\tUi$ as automorphisms of algebras generated by $\tT_j$ $
(j \in \I_\bullet)$ and $\TT_i$ $(i\in\I_{\circ,\tau})$.
\end{proposition}

It is worth noting that the automorhisms $\tT_j$ $
(j \in \I_\bullet)$ and $\TT_i$ $(i\in\I_{\circ,\tau})$ are defined for $\tUi$ over $\Q(q)$, not only for $\tUi_\bF$; see \cite{WZ23}. 
The automorphisms $\TT_i$ on $\tUi$ satisfy 
\begin{align}
\TT_i(x)=\tT_{\bs_i}(\widetilde{\Upsilon}_i)^{-1}\tT_{\bs_i}(x)\tT_{\bs_i}(\widetilde{\Upsilon}_i)
,\qquad \forall x\in\tUi.
    \end{align}
and 
\begin{align}
\label{eq:inverse}
    \TT_i=\sigma^\imath\circ \TT_i^{-1}
\circ\sigma^\imath;
\end{align}
    see \cite[Theorem 6.7]{WZ23}. 

For $j\in\I_\bullet$, we also have $\tT_j(\widetilde{\Upsilon})=\widetilde{\Upsilon}$, and $\tT_j(\widetilde{\Upsilon}_i)=\widetilde{\Upsilon}_i$, for $i\in\I_\circ$; cf. \cite[Proposition 4.13]{BW18b}. Moreover, we have
\begin{align}
    \tT_{j}^{-1}\TT_{i}^{-1}(x)=\TT_i^{-1}\tT_{\tau_{\bullet,\tau}\tau j}^{-1}(x), \qquad \forall i\in\I_{\circ,\tau},j\in\I_\bullet, x\in\tUi.
\end{align}


In this paper, the parameters are required to be {\em balanced}, i.e., if $\vs_i=\vs_{\tau i}$ for any $i\in\I_\circ$.

\begin{lemma}[\text{\cite[Lemma 2.5.1]{Wat21}}] 
\label{lem:iso-paremeter}
For any balanced parameter $\bvs$, there exists an $\bF$-algebra isomorphism $\phi_\bvs:\Ui_{\bvs_\diamond}\rightarrow \Ui_\bvs$ which sends $B_i\mapsto \sqrt{\vs_{\diamond,i}\vs_i^{-1}} B_i$, $E_j\mapsto E_j$, $K_j\mapsto K_j$, $K_\mu\mapsto K_\mu$, for $i\in \I_\circ$, $j\in I_\bullet$, $\mu\in X_\imath$.
\end{lemma}
Since $\bvs$ is a balanced parameter, $\phi_\bvs$ is the restriction of $\Phi_{\ov{\bvs}_\diamond\bvs}$, where $\ov{\bvs}_\diamond\bvs$ is defined by componentwise multiplication with $\ov{\bvs}_\diamond=(\vs_{j,\diamond}^{-1})_{j\in\I_\circ}$.

\section{Root vectors for rank 1 Satake diagrams}
\label{sec:root-rank1}

\begin{table}[h]
\caption{Rank 1 Satake diagrams and local datum}
\label{table:localSatake}
\resizebox{5.5 in}{!}{%
\begin{tabular}{| c | c | c | c |}
\hline
Type & Satake diagram & $\vs_{i,\dm}$ & $\bs_i$
\\
\hline
\begin{tikzpicture}[baseline=0]
\node at (0, -0.15) {AI$_1$};
\end{tikzpicture}

&	
    \begin{tikzpicture}[baseline=0]
		\node  at (0,0) {$\circ$};
		\node  at (0,-.3) {\small 1};
	\end{tikzpicture}
& $\vs_{1,\dm}=-q^{-2}$
& $\bs_1=s_1$
\\
\hline
\begin{tikzpicture}[baseline=0]
\node at (0, -0.15) {AII$_3$};
\end{tikzpicture}
&
   \begin{tikzpicture}[baseline=0]
		\node at (0,0) {$\bullet$};
		\draw (0.1, 0) to (0.4,0);
		\node  at (0.5,0) {$\circ$};
		\draw (0.6, 0) to (0.9,0);
		\node at (1,0) {$\bullet$};
		\node at (0,-.3) {\small 1};
		\node  at (0.5,-.3) {\small 2};
		\node at (1,-.3) {\small 3};
	\end{tikzpicture}
& $\vs_{2,\dm}=-q^{-1}$
& $\bs_2=s_{2132}$
\\
\hline
\begin{tikzpicture}[baseline=0]
\node at (0, -0.2) {AIII$_{11}$};
\end{tikzpicture}
 &
\begin{tikzpicture}[baseline = 6] 
		\node at (-0.5,0) {$\circ$};
		\node at (0.5,0) {$\circ$};
		\draw[bend left, <->] (-0.5, 0.2) to (0.5, 0.2);
		\node at (-0.5,-0.3) {\small 1};
		\node at (0.5,-0.3) {\small 2};
	\end{tikzpicture}
& $\vs_{1,\dm}=-q^{-1}$
& $\bs_1=s_1 s_2$
\\
\hline
\begin{tikzpicture}[baseline=0]
\node at (0, -0.2) {AIV, n$\geq$2};
\end{tikzpicture}
&
\begin{tikzpicture}	[baseline=6]
		\node at (-0.5,0) {$\circ$};
		\draw[-] (-0.4,0) to (-0.1, 0);
		\node  at (0,0) {$\bullet$};
		\node at (2,0) {$\bullet$};
		\node at (2.5,0) {$\circ$};
		\draw[-] (0.1, 0) to (0.5,0);
		\draw[dashed] (0.5,0) to (1.4,0);
		\draw[-] (1.6,0)  to (1.9,0);
		\draw[-] (2.1,0) to (2.4,0);
		\draw[bend left, <->] (-0.5, 0.2) to (2.5, 0.2);
		\node at (-0.5,-.3) {\small 1};
		\node  at (0,-.3) {\small 2};
		\node at (2.5,-.3) {\small n};
	\end{tikzpicture}
& $\vs_{1,\dm}=-q^{-1/2}$
& $\bs_1=s_{1 \cdots  n \cdots 1} $
\\
\hline
\begin{tikzpicture}[baseline=0]
\node at (0, -0.2) {BII, n$\ge$ 2};
\end{tikzpicture} &
    	\begin{tikzpicture}[baseline=0, scale=1.5]
		\node at (1.05,0) {$\circ$};
		\node at (1.5,0) {$\bullet$};
		\draw[-] (1.1,0)  to (1.4,0);
		\draw[-] (1.4,0) to (1.9, 0);
		\draw[dashed] (1.9,0) to (2.7,0);
		\draw[-] (2.7,0) to (2.9, 0);
		\node at (3,0) {$\bullet$};
		\draw[-implies, double equal sign distance]  (3.05, 0) to (3.55, 0);
		\node at (3.6,0) {$\bullet$};
		\node at (1,-.2) {\small 1};
		\node at (1.5,-.2) {\small 2};
		\node at (3.6,-.2) {\small n};
	\end{tikzpicture}	
& $\vs_{1,\dm}=-q_1^{-1}$
& $ \bs_1=s_{1\cdots n \cdots 1}$
\\
\hline
\begin{tikzpicture}[baseline=0]
\node at (0, -0.15) {CII, n$\ge$3};
\end{tikzpicture}
&
		\begin{tikzpicture}[baseline=6]
		\draw (0.6, 0.15) to (0.9, 0.15);
		\node  at (0.5,0.15) {$\bullet$};
		\node at (1,0.15) {$\circ$};
		\node at (1.5,0.15) {$\bullet$};
		\draw[-] (1.1,0.15)  to (1.4,0.15);
		\draw[-] (1.4,0.15) to (1.9, 0.15);
		\draw (1.9, 0.15) to (2.1, 0.15);
		\draw[dashed] (2.1,0.15) to (2.7,0.15);
		\draw[-] (2.7,0.15) to (2.9, 0.15);
		\node at (3,0.15) {$\bullet$};
		\draw[implies-, double equal sign distance]  (3.1, 0.15) to (3.7, 0.15);
		\node at (3.8,0.15) {$\bullet$};
		\node  at (0.5,-0.15) {\small 1};
		\node at (1,-0.15) {\small 2};
		\node at (3.8,-0.15) {\small n};
	\end{tikzpicture}
& $\vs_{2,\dm}=-q_2^{-1/2}$
&	$\bs_2=s_{2\cdots n\cdots 2 1 2\cdots n\cdots 2} $	
\\
\hline
\begin{tikzpicture}[baseline=0]
\node at (0, -0.05) {DII, n$\ge$4};
\end{tikzpicture}&
	\begin{tikzpicture}[baseline=0]
		\node at (1,0) {$\circ$};
		\node at (1.5,0) {$\bullet$};
		\draw[-] (1.1,0)  to (1.4,0);
		\draw[-] (1.4,0) to (1.9, 0);
		\draw[dashed] (1.9,0) to (2.7,0);
		\draw[-] (2.7,0) to (2.9, 0);
		\node at (3,0) {$\bullet$};
		\node at (3.5, 0.4) {$\bullet$};
		\node at (3.5, -0.4) {$\bullet$};
		\draw (3.05, 0.05) to (3.4, 0.39);
		\draw (3.05, -0.05) to (3.4, -0.39);
		\node at (1,-.3) {\small 1};
		\node at (1.5,-.3) {\small 2};
		\node at (3.6, 0.25) {\small n-1};
		\node at (3.5, -0.6) {\small n};
	\end{tikzpicture}		
& $\vs_{1,\dm}=-q^{-1}$
&
\begin{tikzpicture}[baseline=0]
\node at (0, 0.15) { $\bs_1=$};
\node at (0, -0.35) { $ s_{1 \cdots n-2 \cdot n-1 \cdot n \cdot n-2  \cdots 1}$};
\end{tikzpicture}
\\
\hline
\begin{tikzpicture}[baseline=0]
\node at (0, -0.2) {FII};
\end{tikzpicture}
&
\begin{tikzpicture}[baseline=0][scale=1.5]
	\node at (0,0) {$\bullet$};
	\draw (0.1, 0) to (0.4,0);
	\node at (0.5,0) {$\bullet$};
	\draw[-implies, double equal sign distance]  (0.6, 0) to (1.2,0);
	\node at (1.3,0) {$\bullet$};
	\draw (1.4, 0) to (1.7,0);
	\node at (1.8,0) {$\circ$};
	\node at (0,-.3) {\small 1};
	\node at (0.5,-.3) {\small 2};
	\node at (1.3,-.3) {\small 3};
	\node at (1.8,-.3) {\small 4};
\end{tikzpicture}
& $\vs_{4,\dm}= -q_4^{-1/2}$
& $\bs_4=s_{432312343231234}$
\\
\hline
\end{tabular}
}
\newline
\smallskip
\end{table}

Let $(\I=\wI\cup\bI,\tau)$ be a rank $1$ Satake diagram. In the rank 1 case, $\wI=\{i,\tau i\}$ and $\bbw_0=\bs_i$. 
We recall from \cite{WZ23} the list of real rank 1 Satake diagrams and the formulas of $\vs_{i,\diamond}$ and $\bs_i$ for them in Table~\ref{table:localSatake}. Write $\cR^+(\bs_i)=\{\beta_1,\dots,\beta_r\}$.
We shall construct the root vectors $B_{\beta_1},\dots,B_{\beta_r}$ for $\tUi$ in this section, whose leading terms are $F_{\beta_1},\dots, F_{\beta_r}$ (see \eqref{eq:F-root}) respectively. 

\begin{lemma}[\text{see e.g. \cite[Proposition 8.20]{Jan96}}]
\label{lem:Tiwi}
We have $\widetilde{T}_w(F_i) = F_{wi}$, for $i \in \I$ and $w \in W$ such that $wi \in \I$.
\end{lemma}

\subsection{$\un{\text{Type  AI}_1}$}
In this case, 
$\cR^+(\bs_1)=\alpha_1$. Set $B_{\alpha_1}=B_1\in\tUi$. 

\subsection{$\un{\text{Type  AII}_3}$}
In this case, 
$\cR^+(\bs_2)=\{\alpha_2,\alpha_1+\alpha_2,\alpha_2+\alpha_3,\alpha_1+\alpha_2+\alpha_3\}$. Using the reduced expression $\bs_2=s_{2132}$ in Table~\ref{table:localSatake} and Lemma \ref{lem:Tiwi}, we have
\begin{align*}
\begin{split}
&F_{\alpha_2}=F_2,\qquad F_{\alpha_1+\alpha_2}=\tTD_2(F_1)=\tTD_1^{-1}(F_2),
\\
&F_{\alpha_2+\alpha_3}=\tTD_{21}(F_3)=\tTD_3^{-1}(F_2), \qquad F_{\alpha_1+\alpha_2+\alpha_3}= \tTD_{213}(F_2)=\tTD_{13}^{-1}(F_2).
\end{split}
\end{align*}
Set 
\begin{align*}
\begin{split}
B_{\alpha_2}=&B_2,\qquad B_{\alpha_1+\alpha_2}=\tTD_1^{-1}(B_2)=[F_1,B_2]_{q},
\\
B_{\alpha_2+\alpha_3}=&\tTD_3^{-1}(B_2)=[F_3,B_2]_q,
\qquad 
B_{\alpha_1+\alpha_2+\alpha_3}= \tTD_{13}^{-1}(B_2)=\big[F_3,[F_1,B_2]_q\big]_q.
\end{split}
\end{align*}
Obviously, $B_{\beta_i}\in \tUi$ for $\beta_i\in \cR^+(\bs_2)$.


\subsection{$\un{\text{Type  AIII}_{11}}$} 

We have $\cR^+(\bs_1)=\{\alpha_1,\alpha_2\}$. Let $B_{\alpha_1}=B_1$, $B_{\alpha_2}=B_2$.

\subsection{$\un{\text{Type  AIV},n\geq2}$ }

In this case, $\cR^+(\bs_1)=\{\beta_i\mid 1\leq i\leq 2n-1 \}$, where 
\begin{align*}
\beta_i=\begin{cases}
\alpha_1+\alpha_2+\cdots+\alpha_i, & \text{ for }1\leq i\leq n,
\\
\alpha_n+\alpha_{n-1}+\cdots+\alpha_{2n-i+1}, &\text{ for }n+1\leq i\leq 2n-1.
\end{cases}
\end{align*}
Using the reduced expression of $\bs_1$ in Table~\ref{table:localSatake}, we have
\begin{align*}
\begin{split}
&F_{\beta_i}=\tTD_{1\cdots(i-1)}(F_i)=\tTD_{2\cdots i}^{-1}(F_1),\qquad 1\leq i\leq n-1,
\\
&F_{\beta_{n}}=\tTD_{1\cdots(n-1)}(F_n)= [F_n,\tTD^{-1}_{2\cdots(n-1)}(F_{1})]_q,
\\
&F_{\beta_{n+i}}=\tTD_{1\cdots n \cdots (n-i+1)}(F_{n-i})=\tTD_{(n-1)\cdots(n-i+1) }^{-1} (F_n),\qquad 1\leq i\leq n-1.
\end{split}
\end{align*}
We define
\begin{align}
\begin{split}
    B_{\beta_i}=&\tTD_{2\cdots i}^{-1}(B_1),\qquad 1\leq i\leq n-1,
    \\
    B_{\beta_n}=&[B_n,\tTD_{2\cdots(n-1)}^{-1}(B_1)]_q=\bigg[B_n,\Big[\big[[[F_{n-1} ,F_{n-2}]_q \cdots,F_3]_q, F_2\big]_q,B_1\Big]_q\bigg]_q,
    \\
B_{\beta_{n+i}}=&\tTD_{(n-1)\cdots(n-i+1) }^{-1} (B_n),\qquad 1\leq i\leq n-1.
\end{split}
\end{align}

\subsection{$\un{\text{Type BII}, n\geq2}$} 

In this case, $\cR^+(\bs_1)=\{\beta_i\mid 1\leq i\leq 2n-1\}$, where
\begin{align}
\beta_i=\begin{cases}
\alpha_1+\alpha_2+\cdots +\alpha_i,& \text{ for }1\leq i\leq n,
\\
\alpha_1+\cdots +\alpha_{2n-i}+2\alpha_{2n+1-i} +\cdots +2\alpha_n  ,&\text{ for }n+1\leq i\leq 2n-1.
\end{cases}
\end{align}
Using the reduced expression of $\bs_1$ in Table~\ref{table:localSatake}, 
we have

\begin{align}
\begin{split}
&F_{\beta_i}=\tTD_{1\cdots (i-1)}(F_i)=\tTD^{-1}_{2\cdots i}(F_1),\qquad 1\leq i\leq n-1,
\\
&F_{\beta_n}=\tTD_{1\cdots(n-1)}(F_n)=[F_n,\tTD_{2\cdots(n-1)}^{-1}(F_1)]_q=\Big[\cdots\big[[F_n,F_{n-1}]_q,F_{n-2}\big]_q,\cdots,F_1\Big]_q,
\\
&F_{\beta_{n+i}}=\tTD_{1\cdots n\cdots(n-i+1)}(F_{n-i})=\tTD_{2\cdots n\cdots(n-i+1)}^{-1}(F_1),\qquad 1\leq i\leq n-1.
\end{split}
\end{align}

We define 
\begin{align*}
&B_{\beta_i}=\tTD_{2\cdots i}^{-1}(B_1),\qquad 1\leq i\leq n-1,
\\
&B_{\beta_n}=[F_n,\tTD_{2\cdots(n-1)}^{-1}(B_1)]_q=[F_n,B_{\beta_{n-1}}]_q,
\\
&B_{\beta_{n+i}}=\tTD_{2\cdots n\cdots(n-i+1)}^{-1}(B_1),\qquad 1\leq i\leq n-1.
\end{align*}

\subsection{$\un{\text{Type  CII},n\geq3}$ }

$\cR^+(\bs_2)=\{\beta_i\mid 1\leq i\leq 4n-5\}$, where 
\begin{align*}
	\beta_i= \begin{cases}  \alpha_2+\alpha_3+\cdots+\alpha_{i+1}, & \text { for } 1  \leq i \leq n-2, 
		\\ 
		2\alpha_2+2 \alpha_3+\cdots+2 \alpha_i+\alpha_{i+1}, & \text { for } i =n-1,
		\\
		\alpha_2+\cdots+\alpha_{2n-i-1}+2\alpha_{2n-i}+\cdots+2\alpha_{n-1}+\alpha_n, & \text { for } n \leq i \leq 2 n-3 ,
		 \\
		\alpha_1+2 \alpha_2+\cdots+2 \alpha_{n-1}+\alpha_n, & \text { for } i=2 n-2 ,
  \\ \alpha_1+\alpha_2+\cdots+\alpha_{i-2 n+3}, & \text { for } 2n-1 \leq i \leq 3n-4, 
		  \\ 
		2 \alpha_1+\cdots+2 \alpha_{n-1}+\alpha_n, & \text { for } i=3n-3,
		   \\
\alpha_1+\cdots+\alpha_{4n-3-i}+2\alpha_{4n-2-i}+\cdots+2\alpha_{n-1}+\alpha_n, & \text{ for } 3n-2\leq i\leq4n-5,\end{cases}
\end{align*}

Using the reduced expression of $\bs_1$ in Table~\ref{table:localSatake}, 
we have
\begin{align*}
    &F_{\beta_i}=\tTD_{2\cdots i}(F_{i+1})=\tTD_{3\cdots (i+1)}^{-1}(F_2),\qquad 1\leq i\leq n-2,
    \\
    &F_{\beta_{n-1}}=\tTD_{2\cdots (n-1)}(F_{n})=\frac{1}{[2]} \Big[\big[F_n, \tTD_{2\cdots (n-2)}(F_{n-1})\big]_{q^2},\tTD_{2\cdots (n-2)}(F_{n-1})\Big]
    \\
    &\quad\quad \quad=\frac{1}{[2]} \Big[\big[F_n, \tTD_{3\cdots (n-1)}^{-1}(F_{2})\big]_{q^2},\tTD_{3\cdots (n-1)}^{-1}(F_2)\Big],
    \\
    &F_{\beta_{n+i}}=\tTD_{2\cdots n\cdots(n-i)}(F_{n-i-1})=\tTD_{3\cdots n\cdots(n-i)}^{-1}(F_2), \qquad 0\leq i\leq n-3,
    \\
    &{F_{\beta_{2n-2}}=\tTD_{2\cdots n\cdots 2}(F_1)=\big[\tTD_1^{-1}(F_2),\tTD_{2\cdots n\cdots 3}(F_2)\big]_q=\big[\tTD_1^{-1}(F_2), F_{\beta_{2n-3}}\big]_q,}
    \\
     &F_{\beta_{i+2n-1}}=\tTD_{2\cdots n\cdots 21\cdots(i+1)}(F_{i+2})= \tTD_{13\cdots (i+2)}^{-1}(F_2),\qquad 0\leq i\leq n-3,
    \\
    &{F_{\beta_{3n-3}}=\tTD_{2\cdots n\cdots 212\cdots(n-1)}(F_{n})= \frac{1}{[2]} \tTD_{2\cdots n\cdots 212\cdots(n-2)}\Big(\big[[F_n,F_{n-1}]_{q^2},F_{n-1}\big]\Big)}
    \\
    &\qquad \qquad{=\frac{1}{[2]}\big[[F_n,F_{\beta_{3n-2}}]_{q^2},F_{\beta_{3n-2}}\big],}
    \\
    &F_{\beta_{i+3n-2}}=\tTD_{2\cdots n\cdots 21\cdots n\cdots (n-i)}(F_{n-i-1})=\tTD_{13\cdots n\cdots(n-i)}^{-1}(F_2),
    \qquad 0\leq i\leq n-3.
\end{align*}
We define
\begin{align} 
\begin{split}
&B_{\beta_i}= \tTD_{3\cdots (i+1)}^{-1}(B_2),\qquad 1\leq i\leq n-2,
    \\
    &B_{\beta_{n-1}} =\frac{1}{[2]} \Big[\big[F_n, B_{\beta_{n-2}}\big]_{q^2},B_{\beta_{n-2}}\Big],
    \\
    &B_{\beta_{n+i}}= \tTD_{3\cdots n\cdots(n-i)}^{-1}(B_2), \qquad 0\leq i\leq n-3,
    \\
    &B_{\beta_{2n-2}}=\big[\tTD_1^{-1}(B_2), B_{\beta_{2n-3}}\big]_q,
    \\
    &B_{\beta_{i+2n-1}}=  \tTD_{13\cdots (i+2)}^{-1}(B_2),\qquad 0\leq i\leq n-3,
    \\
    &B_{\beta_{3n-3}} =\frac{1}{[2]}\big[[F_n,B_{\beta_{3n-2}}]_{q^2},B_{\beta_{3n-2}}\big],
    \\
    &B_{\beta_{i+3n-2}} = \tTD_{13\cdots n\cdots(n-i)}^{-1}(B_2),
    \qquad 0\leq i\leq n-3.
\end{split}
\end{align}
\subsection{$\un{\text{Type  DII},n\geq4}$ }
In this case, $\cR^+(\bs_1)=\{\beta_i\mid 1\leq i\leq 2n-2\}$ where
\begin{align*}
\beta_i=
\begin{cases}
\alpha_1+\alpha_2+\cdots+\alpha_i, & \text{ for }1\leq i\leq n-1,
\\
\alpha_1+\alpha_2+\cdots+\alpha_{n-2}+\alpha_n , &\text{ for } i=n,
\\
\alpha_1+\cdots+\alpha_{2n-i-1}+2\alpha_{2n-i}+\cdots+2\alpha_{n-2} +\alpha_{n-1}+\alpha_n, &\text{ for } n+1\leq i\leq 2n-2.
\end{cases}
\end{align*}
Using the reduced expression of $\bs_1$ in Table~\ref{table:localSatake}, we have
\begin{align}
\begin{split}
&F_{\beta_i}=\tTD_{1\cdots(i-1)}(F_i)=\tTD_{2\cdots i}^{-1}(F_1),\qquad 1\leq i\leq n-1,
\\
&F_{\beta_{n}}=\tTD_{1\cdots(n-1)}(F_n)= \tTD_{2\cdots (n-2)\cdot n}^{-1}(F_1),
\\
&F_{\beta_{n+1}}=\tTD_{1\cdots n}(F_{n-2})=\tTD_{2\cdots n}^{-1}(F_1),
\\
&F_{\beta_{n+i}}=\tTD_{1\cdots n\cdot (n-2)\cdot(n-3)\cdots(n-i)}(F_{n-i-1})=\tTD_{2\cdots n\cdot (n-2)\cdot (n-3)\cdots (n-i)}^{-1}(F_1),\qquad 2\leq i\leq n-2.
\end{split}
\end{align}
We define 
\begin{align}
\begin{split}
&B_{\beta_i}=\tTD_{2\cdots i}^{-1}(B_1),\qquad 1\leq i\leq n-1,
\\
&B_{\beta_{n}}= \tTD_{2\cdots (n-2)\cdot n}^{-1}(B_1),
\\
&B_{\beta_{n+1}}=\tTD_{2\cdots n}^{-1}(B_1),
\\
&B_{\beta_{n+i}}=\tTD_{2\cdots n\cdot (n-2)\cdot (n-3)\cdots (n-i)}^{-1}(B_1),\qquad 2\leq i\leq n-2.
\end{split}
\end{align}

\subsection{$\un{\text{Type  }F_4}$ }

$\cR^+(\bs_4)=\{\beta_i\mid 1\leq i\leq 15\}$, where 
\begin{align*}
&\beta_1=\alpha_4,\quad \beta_2=\alpha_3+\alpha_4,\quad \beta_3=\alpha_2+2\alpha_3+2\alpha_4,\quad \beta_4=\alpha_2+\alpha_3+\alpha_4,\\
&\beta_5=\alpha_1+\alpha_2+2\alpha_3+2
\alpha_4,\quad \beta_6=\alpha_1+2\alpha_2+2\alpha_3+2
\alpha_4,\quad \beta_7=\alpha_1+\alpha_2+\alpha_3+
\alpha_4,
\\
&\beta_8=\alpha_1+2\alpha_2+3\alpha_3+2
\alpha_4,\quad \beta_9=\alpha_2+2\alpha_3+
\alpha_4,\quad \beta_{10}=\alpha_1+2\alpha_2+4\alpha_3+2
\alpha_4,\\
&\beta_{11}=\alpha_1+\alpha_2+2\alpha_3+
\alpha_4,\quad \beta_{12}=\alpha_1+3\alpha_2+4\alpha_3+2
\alpha_4,\\
&\beta_{13}=2\alpha_1+3\alpha_2+4\alpha_3+2
\alpha_4,\quad \beta_{14}=\alpha_1+2\alpha_2+2\alpha_3+
\alpha_4,\\
& \beta_{15}=\alpha_1+2\alpha_2+3\alpha_3+
\alpha_4.
\end{align*}
We have 
\begin{align*}
    &F_{\beta_1}=F_4,\qquad F_{\beta_2}=\tTD_4(F_3)=\tTD_3^{-1}(F_4),\qquad    
    \\
    &F_{\beta_3}=\tTD_{43}(F_2)=\frac{1}{[2]}\Big[\big[F_2,[F_3,F_4]_q\big]_{q^2},[F_3,F_4]_q\Big],\qquad
   F_{\beta_4}=\tTD_{432}(F_3)=\tTD_{32}^{-1}(F_4),\\
   &F_{\beta_5}=\tTD_{4323}(F_1)=\tTD_{1}^{-1}\tTD_{43}(F_2)=\tTD_1^{-1}(F_{\beta_3}),
    \\
    &F_{\beta_6}=\tTD_{43231}(F_2)={\tTD_{12}^{-1}\tTD_{34}^{-1}(F_2)}=[F_2,F_{\beta_5}]_q,\qquad 
    F_{\beta_7}=\tTD_{432312}(F_3)=\tTD_{321}^{-1}(F_4),
    \\
    &{F_{\beta_8}=\tTD_{4323123}(F_4)=[\tTD_{323}^{-1}(F_4),F_{\beta_7}]_q}, \qquad F_{\beta_9}= \tTD_{43231234}(F_3)=\tTD_{323}^{-1}(F_4),
    \\
    &F_{\beta_{10}}=\tTD_{432312343}(F_2)={\frac{1}{[2]}\big[[F_1,F_{\beta_9}]_{q^2},F_{\beta_9}\big]},
    \\
    &F_{\beta_{11}}=\tTD_{4323123432}(F_3)=\tTD_{3213}^{-1}(F_4), \qquad F_{\beta_{12}}={\tTD_{43231234323}(F_1)=[F_2,F_{\beta_{10}}]_q},
    \\
    &F_{\beta_{13}}={\tTD_{432312343231}(F_2)=\big[F_1,F_{\beta_{12}}\big]_q},
    \qquad F_{\beta_{14}}=\tTD_{4323123432312}(F_3)={\tTD_{32312}^{-1}(F_4)},
    \\
    &F_{\beta_{15}}=\tTD_{43231234323123}(F_4)=\tTD_{323123}^{-1}(F_4).
\end{align*}
We define 
\begin{align}
\begin{split}
&B_{\beta_1}=B_4,\qquad B_{\beta_{2}}=\tTD_{3}^{-1}(B_4),\qquad B_{\beta_{3}}=\frac{1}{[2]}\big[[F_2,B_{\beta_2}]_{q^2},B_{\beta_2}\big]
\\
&B_{\beta_{4}}=\tTD_{32}^{-1}(B_4),\qquad
B_{\beta_{5}}=\tTD_{1}^{-1}(B_{\beta_3})=[F_1,B_{\beta_3}]_q,\qquad
B_{\beta_{6}}=[F_2,B_{\beta_5}]_q,\qquad
\\
&B_{\beta_{7}}=\tTD_{321}^{-1}(B_4),\qquad
B_{\beta_{8}}=[\tTD_{323}^{-1}(B_4),B_{\beta_7}]_q,\qquad
B_{\beta_{9}}=\tTD_{323}^{-1}(B_4),\qquad
\\
&B_{\beta_{10}}=\frac{1}{[2]}\big[[F_1,B_{\beta_9}]_{q^2},B_{\beta_9}\big],\qquad
B_{\beta_{11}}=\tTD_{3213}^{-1}(B_4),\qquad
B_{\beta_{12}}=[F_2,B_{\beta_{10}}]_q,\qquad
\\
&B_{\beta_{13}}=[F_1,B_{\beta_{12}}]_q,\qquad
B_{\beta_{14}}=\tTD_{32312}^{-1}(B_4),\qquad
B_{\beta_{15}}=\tTD_{323123}^{-1}(B_4),\qquad
\end{split}
\end{align}

\begin{theorem}
\label{thm:rkone}
Let $(\I=\wI\cup\bI,\tau )$ be a Satake diagram of real rank 1 and $(\tU,\tUi)$ be the corresponding quantum symmetric pair. Write $\wI=\{i,\tau i\}$ and $\cR^+(\bs_i)=\{\beta_1,\cdots,\beta_r\}$. For each $1\leq k\leq r$, let $B_{\beta_k}$ be the root vectors in $\tUi$ constructed above. Then the leading term of $B_{\beta_k}$ is $F_{\beta_k}$.
\end{theorem}

\begin{proof}

The statement follows from the explicit construction of root vectors $B_{\beta_k}$.  

\end{proof}

\begin{remark}
We have $\phi(F_{\beta_k})=B_{\beta_k}$ where $\phi:\tU^-\otimes \tU^+_\bullet\otimes \tU^{\imath,\gr}\rightarrow \tU^{\imath,\gr}$ is the algebra homomorphism in \eqref{eq:grade}. 
\end{remark}

\section{PBW bases for $\imath$quantum groups}
\label{sec:PBW-Ui}

Now we consider a Satake diagram $(\I=\wI\cup\bI,\tau )$ of higher rank. Let $i\in \wI$. Recall the real rank 1 subdiagram $(\I_{\bullet,i}=\{i,\tau i\}\cup\I_\bullet,\tau)$. Let $\Ui_{[i]}$ be the real rank 1 $\imath$quantum group associated to the Satake diagram is $(\I_{\bullet,i}=\{i,\tau i\}\cup\I_\bullet,\tau)$. Clearly, $\tUi_{[i]}=\tUi_{[\tau i]}$. The algebra $\tUi_{[i]}$ is naturally identified with the subalgebra of $\tUi$ generated by $B_i,B_{\tau i},k_i,k_{\tau i}$ and $\tU_\bullet$. By the construction in Section~\ref{sec:root-rank1}, for any $i\in \wI,\beta\in \cR^+(\bs_i)$, we have the root vectors $B_\beta\in \tUi_{[i]}$.

Let $\bbw_0=\bs_{i_1}\bs_{i_2}\cdots \bs_{i_\ell}$ be a fixed reduced expression for the longest element $\bbw_0$ in $W^\circ$. For any $\beta\in \cR^+(\bbw_0)$, there exists a unique $1\leq j\leq \ell$ and a unique $\beta_0\in\cR^+(\bs_{i_j})$ such that $\beta=\bs_{i_1}\bs_{i_2}\cdots \bs_{i_{j-1}}(\beta_0)$. In this way, define 
\begin{align}\label{def:Bbeta}
    B_\beta:=\TT_{i_1}\TT_{i_2}\cdots \TT_{i_{j-1}}(B_{\beta_0}), \qquad \text{ for }\beta\in \cR^+(\bbw_0),
\end{align}
where $B_{\beta_0}\in \tUi_{[i_j]}$ is the root vector defined in Section~\ref{sec:root-rank1} for $\beta_0\in \cR^+(\bs_{i_j})$. 

\begin{proposition}
If $\beta=\alpha_i$ for some $i\in\wI$, then $B_\beta=B_i$.
\end{proposition}

\begin{proof}
Write $\beta=\bs_{i_1}\bs_{i_2}\cdots \bs_{i_{j-1}}(\beta_0)$ for $\beta_0\in\cR^+(\bs_{i_j})$. We claim that $\beta_0$ must be $\alpha_{i_j}$ or $\alpha_{\tau i_j}$. Otherwise, write $\beta_0=c \alpha_{i_j} + c' \alpha_{\tau i_j} +\sum_{r\in \bI} c_r\alpha_{r} $ where $c,c',c_r$ are coefficients, and then we must have $c_r\neq 0$ for some $r\in \bI$. Recall that $\bs_{i_k}=w_{\bullet,i_k}\bw $; then $\bs_{i_k}$ permutes $\bI$. This implies that, if we write $\beta$ in terms of simple roots, the multiplicity of $\alpha_{r'}$ must be nonzero for some $r'\in \bI$, which contradicts the assumption of $\beta$. 

By our rank 1 construction in Section~\ref{sec:root-rank1}, $B_{\alpha_{i_j}}=B_{i_j}$ and $B_{\alpha_{\tau i_j}}=B_{\tau i_j}$. Using the claim and \cite[Theorem~7.13]{WZ23}, we conclude that $B_\beta=B_i$ as desired.
\end{proof}

By using the reduced expression of $\bs_{i_j}$ as Table \ref{table:localSatake} for $1\leq j\leq \ell$, we obtain a reduced expression $\bbw_0=s_{k_1}s_{k_2}\cdots s_{k_{t}}$ in $W$. 
This reduced expression of $\bbw_0$ induces a total order on $\cR^+(\bbw_0)$. We define, for any $\ba=(a_\beta)_{\beta\in\cR^+(\bbw_0)}\in\N^{\ell(\bbw_0)}$), 
\begin{align}
    B^{\ba}:=\prod_{\beta\in\cR^+(\bbw_0)} B_{\beta}^{a_\beta}.
\end{align}
For any $\beta\in\cR^+(\bbw_0)$, we define $F_{\beta}$ in the same way as \eqref{eq:F-root} by using this reduced expression of $\bbw_0$, and define $F^{\ba}=\prod_{\beta\in\cR^+(\bbw_0)} F_{\beta}^{a_\beta}$. 

\begin{lemma}
\label{lem:leadingPBW}
Keep notations as above. For  $\beta\in\cR^+(\bbw_0)$, the leading term of $B_\beta$ in $\tU$ is $F_\beta$. 
\end{lemma}

\begin{proof}
By definition \eqref{def:Bbeta}, we know 
\begin{align*}
    B_\beta=\TT_{i_1}\TT_{i_2}\cdots \TT_{i_{j-1}}(B_{\beta_0}), \qquad \text{ for unique }\beta_0\in \cR^+(\bs_{i_j}).
\end{align*}
If $j=1$, the result follows from Theorem~\ref{thm:rkone}. 

For $j>1$, in order to use \eqref{eq:compTT},  we shall consider $\sigma^\imath(B_\beta)$, which equals to $\TT_{i_1}^{-1}\TT_{i_2}^{-1}\cdots \TT_{i_{j-1}}^{-1}(\sigma^\imath(B_{\beta_0}))$ by \eqref{eq:inverse}. It suffices to show that the leading term of $\sigma^\imath(B_\beta)$ is $\sigma(F_\beta)$. Then the desired statement follows from the definitions of $\sigma^\imath$ and $\sigma$. 

Using \eqref{eq:compTT}, we have 
\begin{align}
\TT_i^{-1}(x)  =\tfX_i \cdot \tT_{\bs_i}^{-1}(x)\cdot \tfX_i^{-1},\qquad \forall x\in \tUi.
\end{align}
Let $\bbw=\bs_{i_1}\cdots\bs_{i_{j}}$ and $\tfX_{\bbw}$ be the partial quasi $K$-matrix defined by 
\begin{align}
\tfX_{\bbw}=\widetilde{\fX}_{i_1}\cdot \tT_{\bs_{i_1}}^{-1}(\widetilde{\fX}_{i_2})\cdots \tT_{\bs_{i_1}}^{-1}\cdots \tT_{\bs_{i_{j-1}}}^{-1}(\widetilde{\fX}_{i_{j}}).
\end{align}
(here $\tfX_w$ differs from those in \cite[Definition 3.13]{DK19} and \cite[Section 8.1]{WZ23} by $\sigma$.) Using \eqref{eq:compTT}, we have
\begin{align} 
\notag
\sigma^\imath(B_\beta)&=\TT_{i_1}^{-1}\TT_{i_2}^{-1}\cdots \TT_{i_{j-1}}^{-1}\big(\sigma^\imath(B_{\beta_0})\big)
\\\label{eq:lower}
&=\tfX_{\bbw} \cdot \tT^{-1}_{\bs_{i_1}}\tT^{-1}_{\bs_{i_2}}\cdots \tT_{\bs_{i_{j-1}}}^{-1}\big(\sigma^\imath(B_{\beta_0})\big) 
\cdot\tfX_{\bbw}^{-1}.
\end{align}

By Proposition~\ref{prop:Kmatrix}, for $i\in \wI$, $\tfX_i=\sum_{\mu\in \N \I_{\bullet,i}} \tfX_i^\mu$ where $\tfX_i^0=1,\tfX_i^\mu\in \tU^+_\mu$. Since $\bbw= \bs_{i_1}\bs_{i_2}\cdots \bs_{i_{j}}\in W^\circ $ is a reduced expression, the component $\tT_{\bs_{i_1}}^{-1}\cdots \tT_{\bs_{i_{k-1}}}^{-1}(\widetilde{\fX}_{i_{k}}^\mu)$ of a factor for $\tfX_{\bbw}$ lies in $ \tU^+$ for any $1\leq k\leq j$.
Hence, $\tfX_{\bbw}$ admits the form 
\begin{align}
\label{eq:tfXw+}
\tfX_{\bbw}=\sum_{\mu\in\N\I}\tfX_{\bbw}^\mu,\qquad  \tfX_{\bbw}^0=1,\;\tfX_{\bbw}^\mu\in \tU^+_\mu.
\end{align}
Using 
\eqref{eq:lower} and \eqref{eq:tfXw+}, we know that
$$\sigma^\imath(B_\beta)=\TT_{i_1}^{-1}\TT_{i_2}^{-1}\cdots \TT_{i_l}^{-1}\big(\sigma^\imath(B_{\beta_0})\big) \in \tT^{-1}_{\bs_{i_1}}\tT^{-1}_{\bs_{i_2}}\cdots \tT_{\bs_{i_{j-1}}}^{-1}\big(\sigma^\imath(B_{\beta_0})\big)
+\sum_{\mu<\beta} \tU_\mu.$$
i.e., the leading term of $\sigma^\imath(B_\beta)$ is the same as the leading term of $\tT^{-1}_{\bs_{i_1}}\tT^{-1}_{\bs_{i_2}}\cdots \tT_{\bs_{i_{j-1}}}^{-1}\big(\sigma^\imath(B_{\beta_0})\big)$.

{\bf Claim:}  $\tT^{-1}_{\bs_{i_1}}\tT^{-1}_{\bs_{i_2}}\cdots \tT_{\bs_{i_{j-1}}}^{-1}\big(\sigma^\imath(B_{\beta_0})\big)\in \sigma(F_\beta)+\sum_{\mu< \beta} \tU_\mu$.

If the claim holds, then
the leading term of $\sigma^\imath(B_\beta)$ is $\sigma(F_\beta)$ as desired. 

It remains to prove the claim.
Note that $i_j\in\I_\circ$, and $B_{i_j}=F_{i_j}+\tTD_{\bw}( E_{\tau i_j}) K_{i_j}'$. By the construction of $B_{\beta_0}$ for $\beta_0\in \cR^+(\bs_{i_j})$ in Section \ref{sec:root-rank1}, we know that \begin{align}
\label{eq:Bgamma}
    B_{\beta_0}\in F_{\beta_0}+\sum_{\mu\in I_{\beta_0}} \tU_\mu,
\end{align}
where 
$$I_{\beta_0}=\{{\beta_0}-aw_\bullet(\alpha_{ i_j})-bw_\bullet(\alpha_{ \tau i_j})\mid a,b\in\N,\text{ either }a\neq 0 \text{ or }b\neq0\}.$$ 
Let $\bbw'=\bs_{i_1}\bs_{i_2}\cdots\bs_{i_{j-1}}$. 
Note that $\bbw'(\alpha_{i_j}),\bbw'(\alpha_{\tau i_j})\in\cR^+\setminus\cR^+_\bullet$ since $\bbw_0= \bs_{i_1}\bs_{i_2}\cdots \bs_{i_\ell}$ is a reduced expression. Since $\bbw'(\bw(\alpha_{i_j}))=w_\bullet\bbw'(\alpha_{i_j})\in w_\bullet(\cR^+\setminus\cR^+_\bullet)\subseteq \cR^+$,
we know $\bbw'(\mu)\leq \bbw'(\beta_0)=\beta$ for any $\mu\in I_{\beta_0}$, and the equality holds if and only if $\mu=\beta_0$. 
Using \eqref{eq:Bgamma}, we know 
$$\tT^{-1}_{\bs_{i_1}}\tT^{-1}_{\bs_{i_2}}\cdots \tT_{\bs_{i_{j-1}}}^{-1}\big(\sigma^\imath(B_{\beta_0})\big)\in \tT^{-1}_{\bs_{i_1}}\tT^{-1}_{\bs_{i_2}}\cdots \tT_{\bs_{i_{j-1}}}^{-1}\big(\sigma(F_{\beta_0})\big)+\sum_{\mu<\beta} \tU_\mu.$$
Note that $\sigma(F_\beta)=\tT^{-1}_{\bs_{i_1}}\tT^{-1}_{\bs_{i_2}}\cdots \tT_{\bs_{i_{j-1}}}^{-1}\big(\sigma(F_{\beta_0})\big)$. The claim follows.
\end{proof}

The following corollary follows from Lemma \ref{lem:leadingPBW}. 

\begin{corollary}
\label{cor:Ba}
 For  $\beta\in\cR^+(\bbw_0),\ba\in\N^{\ell(\bbw_0)}$, the leading term of $B_\beta^{\ba}$ is $F_\beta^\ba$ in $\tU$.
\end{corollary}

Given a reduced expression of $w_\bullet=s_{j_1}s_{j_2}\cdots s_{j_r}$ of the longest element $w_\bullet$ in $W_\bullet$. Then $\cR^+_\bullet=\cR^+(w_\bullet)$, and $r=|\cR^+_{\bullet}|$. For any $\gamma\in \cR^+_\bullet$, by using $\widetilde{T}_i$ ($i\in\I_\bullet$), we can define $F_\gamma$ and $E_\gamma$ in $\tU_\bullet$ as Proposition \ref{prop:PBW-QG}. Moreover, given a total order on $\cR^+_\bullet$, we define for any $\bc=(c_{\gamma})_{\gamma\in\cR^+_\bullet}\in\N^{r}$ that 
\begin{align}
F_\bullet^{\bc}:=\prod_{\gamma\in\cR^+_\bullet}F_\gamma^{c_\gamma},\qquad E_\bullet^{\bc}:=\prod_{\gamma\in\cR^+_\bullet}E_\gamma^{c_\gamma}.
\end{align}
Then we have the monomials $F_\bullet^{\bc}E_\bullet^{\bd}$ ($\bc,\bd\in \N^{r}$) form a $\tU^0_\bullet$-basis of $\tU_\bullet$; cf. Proposition \ref{prop:PBW-QG}.  

\begin{theorem}
\label{thm:iPBW}
Let $\bbw_0=\bs_{i_1}\bs_{i_2}\cdots\bs_{i_\ell}$ be a reduced expression for the longest element $\bbw_0$ in $W^\circ$, and $\bw=s_{j_1}s_{j_2}\cdots s_{j_r}$ a reduced expression of the longest element $w_\bullet$ in $W_\bullet$. Then the monomials 
$B^\ba F_\bullet^{\bc} E_\bullet^{\bd}$ ($\ba\in \N^{\ell(\bbw_0)},\bc,\bd\in\N^{r}$) 
form a $\tU^{\imath0}$-basis of $\tUi$.
\end{theorem}

\begin{proof}

Using Lemma \ref{lem:leadingPBW}, we know that the leading term of $B^{\ba}$ is $F^\ba$ for any $\ba\in\N^{\ell(\bbw_0)}$. 

Recall the reduced expression $\bbw_0=s_{k_1}s_{k_2}\cdots s_{k_{t}}$ in $W$, obtained using the reduced expressions for $\bs_i$. Since $\ell(w_0)=\ell(\bbw_0)+\ell(\bw)$, we obtain a reduced expression of $w_0$ by composing $\bbw_0=s_{k_1}s_{k_2}\cdots s_{k_{t}}$ and $\bw=s_{j_1}s_{j_2}\cdots s_{j_r}$. Let $\tau_0$ be the diagram involution associated to $w_0$. By Lemma~\ref{lem:positive-roots-rankone} and its proof, elements
\begin{align*}
F^\ba \tTD_{\bbw_0}(F_\bullet^{\bc})
=F^\ba \tau\tau_0 (F_\bullet^{\bc})\quad (\ba\in \N^{\ell(\bbw_0)},\bc\in\N^{r}) 
\end{align*}
form a PBW basis of $\tU^-$ with respect to the above reduced expression of $w_0$. It follows that $F^\ba F_\bullet^{\bc}$ ($\ba\in \N^{\ell(\bbw_0)},\bc\in\N^{r}$)  form a PBW basis of $\tU^-\cong \U^-$ with respect to the reduced expression $w_0=s_{k_1}s_{k_2}\cdots s_{k_{t}}s_{j'_1}s_{j'_2}\cdots s_{j'_r}$ where $j_s'=\tau\tau_0(j_s)$. 

Recall the isomorphism $\phi:\tU^- \otimes \tU_{\bullet}^+ \otimes \U^{\imath,0}\rightarrow \U^{\imath,\gr}$ defined in 
\eqref{eq:grade}. Since $E_\bullet^{\bd}$ form a PBW basis for $\tU^+_\bullet$, elements $F^\ba F^{\bc}_\bullet \otimes E_\bullet^{\bd} \otimes 1$ ($\ba\in \N^{\ell(\bbw_0)},\bc,\bd\in\N^{r}$) form a right $\tU^{\imath0}$-basis for $\tU^- \otimes \tU_{\bullet}^+\otimes \tU^{\imath0}$. By definition, $\phi(F^\ba F^{\bc}_\bullet\otimes E_\bullet^{\bd} \otimes 1)=B^\ba F^{\bc}_\bullet E_\bullet^{\bd}$. Therefore, the monomials 
$B^\ba F_\bullet^{\bc} E_\bullet^{\bd}$ ($\ba\in \N^{\ell(\bbw_0)},\bc,\bd\in\N^{r}$) form a $\tU^{\imath0}$-basis for $ \tU^{\imath,\gr}$, and hence they form a $\tU^{\imath0}$-basis of $\tUi$. 

\end{proof}

For any parameters $\bvs\in(\Q(q)^\times)^{\I_\circ}$, there is a canonical homomorphism $\pi_\bvs^\imath:\tUi\rightarrow \Ui_\bvs$ defined in \eqref{def:pibvs}. Using this morphism $\pi_\bvs^\imath$, we obtain the root vectors $B_\beta$ in $\Ui_\bvs$ for any $\beta\in\cR^+(\bbw_0)$, and then the monomials 
$B^\ba F_\bullet^{\bc} E_\bullet^{\bd}$ ($\ba\in \N^{\ell(\bbw_0)},\bc,\bd\in\N^{r}$) in $\Ui_\bvs$.
We have the following corollary. 
\begin{corollary}
\label{cor:PBW-Ui}
Let $\bbw_0=\bs_{i_1}\bs_{i_2}\cdots\bs_{i_\ell}$ be a reduced expression for the longest element $\bbw_0$ in $W^\circ$, and $w_\bullet=s_{j_1}s_{j_2}\cdots s_{j_r}$ a reduced expression of the longest element $w_\bullet$ in $W_\bullet$. Then the monomials 
$B^\ba F_\bullet^{\bc} E_\bullet^{\bd}$ ($\ba\in \N^{\ell(\bbw_0)},\bc,\bd\in\N^{r}$) 
form a $\U^{\imath0}$-basis of $\Ui_\bvs$.
\end{corollary}

\section{Modified $\imath$quantum groups}
\label{sec:Uidot}
In this section, we review the modified quantum groups and the modified $\imath$quantum groups. We construct $\imath$divided powers for root vectors of modified $\imath$quantum groups in the rank 1 case.

\subsection{The algebra $\dot{\U}$}

Let $(Y,X,\langle \cdot,\cdot\rangle,...)$ be a root datum of type $(\I,\cdot)$; see \S\ref{sec:prelim}. 

For $\lambda',\lambda''\in X$, define
\begin{align}
    {}_{\lambda'}\U_{\lambda''}:=\U/\big(\sum_{\mu\in Y}(K_\mu-q^{\langle\mu,\lambda'\rangle})\U+\sum_{\mu\in Y}\U(K_\mu-q^{\langle \mu,\lambda''\rangle})\big).
\end{align}
Lusztig's modified quantum group is 
\begin{align}
    \dot{\U}=\bigoplus_{\lambda',\lambda''\in X}({}_{\lambda'}\U_{\lambda''}),
\end{align}
which is an associative $\Q(q)$-algebra inherited from that of $\U$.
Let $\pi_{\lambda',\lambda''}:\U\rightarrow{}_{\lambda'}\U_{\lambda''} $ be the canonical projection, and $\mathbf{1}_\lambda =\pi_{\lambda,\lambda}(1)$ for any $\lambda\in X$. Then $\mathbf{1}_\lambda$ ($\lambda\in X$) are the orthogonal idempotents, and 
\begin{align}
{}_{\lambda'}\U_{\lambda''}=\mathbf{1}_{\lambda'}\dot{\U}\mathbf{1}_{\lambda''}.
\end{align}

Recall $\A=\Z[q,q^{-1}]$. 
The $\A$-form $\dot{\U}_\A$ of $\dot{\U}$ is the $\A$-subalgebra of $\dot{\U}$ generated by the elements $E_i^{(n)}\mathbf{1}_{\lambda}$, $F_i^{(n)}\mathbf{1}_\lambda$ for $i\in\I$, $n\geq0$, and $\lambda\in X$; see \cite[Lemma 23.2.2~(c)]{Lus93}.

\begin{proposition}[\text{cf. \cite[Lemma 23.2.2]{Lus93}}]
\label{prop:PBWUdot}
Fix a reduced expression ${w_0}=s_{i_1} \cdots s_{i_l}$ for the longest element $w_0\in W$. Set
\begin{align*}
F_{\beta_k}^{(a)}=\widetilde{T}_{i_1}\widetilde{T}_{i_2}\cdots \widetilde{T}_{i_{k-1}}(F_{i_k}^{(a)}),
\\
E_{\beta_k}^{(a)}=\widetilde{T}_{i_1}\widetilde{T}_{i_2}\cdots \widetilde{T}_{i_{k-1}}(E_{i_k}^{(a)}),
\end{align*}
for $1\leq k\leq l,a\in \N$. Then the monomials $F_{\beta_1}^{(a_1)} F_{\beta_2}^{(a_2)}\cdots F_{\beta_l}^{(a_l)} \mathbf{1}_\lambda E_{\beta_1}^{(b_1)} E_{\beta_2}^{(b_2)}\cdots E_{\beta_l}^{(b_l)}$, $a_1,\ldots, a_l,b_1,b_2,\dots,b_l\in \N,\lambda\in X$ form an $\A$-basis of $\dot{\U}_\A$.
\end{proposition}

\begin{proof}
This proposition follows from Proposition \ref{prop:PBW-QG} by using \cite[Lemma 23.2.2]{Lus93}.
\end{proof}

Let $M(\lambda)$ be the Verma module of $\U$ with highest weight $\lambda\in X$ and with a highest weight vector $\eta_\lambda$. Denote by $L(\lambda)$ the unique simple quotient $\U$-module of $M(\lambda)$. For $\lambda\in X^+$, set $_\A L(\lambda)={}_\A \U^-\eta_\lambda$ the $\A$-submodule of $L(\lambda)$.

\subsection{The algebra $\dot{\U}^\imath$}
Recall the $\imath$quantum group $\Ui=\Ui_\bvs$ over the field $\bF$ from section~\ref{sec:QSP} and $X_\imath,Y^\imath$ from \eqref{def:XiYi}. 
Following \cite{BW18b,BW21}, for any $\lambda',\lambda''\in X_\imath$, we define
\begin{align}
\label{def:Uilambda}
    {}_{\lambda'}\Ui_{\lambda''}
    :=\Ui/\big(\sum_{\mu\in Y^\imath}(K_\mu-q^{\langle\mu,\lambda'\rangle})\Ui+\sum_{\mu\in Y^\imath}\Ui(K_\mu-q^{\langle \mu,\lambda''\rangle})\big).
\end{align}
Define
\begin{align}
    \dot{\U}^\imath=\bigoplus_{\lambda',\lambda''\in X_\imath}({}_{\lambda'}\Ui_{\lambda''}),
\end{align}
which is an associative $\bF$-algebra. Denote  
$\pi_{\lambda',\lambda''}:\Ui\rightarrow{}_{\lambda'}\Ui_{\lambda''} $ the canonical projection, and $\mathbf{1}_\lambda =\pi_{\lambda,\lambda}(1)$ for any $\lambda',\lambda'',\lambda\in X_\imath$. Then $\mathbf{1}_\lambda$ ($\lambda\in X_\imath$) are the orthogonal idempotents, and
\begin{align}
{}_{\lambda'}\Ui_{\lambda''}=\mathbf{1}_{\lambda'}\dot{\U}^\imath\mathbf{1}_{\lambda''}.
\end{align}

\begin{proposition}
\label{prop:PBW-Uifiled}
    Let $\bbw_0=\bs_{i_1}\bs_{i_2}\cdots\bs_{i_\ell}$ be a reduced expression for the longest element $\bbw_0$ in $W^\circ$, and $w_\bullet=s_{j_1}s_{j_2}\cdots s_{j_r}$ a reduced expression of the longest element $w_\bullet$ in $W_\bullet$. Then the monomials 
$B^\ba F_\bullet^{\bc} E_\bullet^{\bd}\mathbf{1}_\zeta$ ($\ba\in \N^{\ell(\bbw_0)},\bc,\bd\in\N^{r}$, $\zeta\in X_\imath$) 
form an $\bF$-basis of $\dot{\U}^{\imath}$.
\end{proposition}

\begin{proof}
    It follows from the definition of $\dot{\U}^\imath$ and Corollary \ref{cor:PBW-Ui}. 
\end{proof}

The main goal of this and next section is to construct a PBW type basis for suitable integral form of $\dot{\U}^\imath$. In our construction of the integral PBW type basis, an important role is played by the relative braid group symmetries, whose construction involves the distinguished parameter.
Since the distinguished parameters $\bvs_\diamond$ in Table \ref{table:localSatake} involve $q^{1/2}$, we need to enlarge the ring $\A$. Let \begin{align}\bar{\A}=\Z[v,v^{-1}]\text{ where }v=q^{1/2}.
\end{align}
Note that $\A\subseteq \bar{\A}$ as a subring.
The $\bar{\A}$-form $\dot{\U}_{\bar{\A}}$ of $\dot{\U}$ is the $\bar{\A}$-subalgebra of $\dot{\U}$ generated by the elements $E_i^{(n)}\mathbf{1}_{\lambda}$, $F_i^{(n)}\mathbf{1}_\lambda$ for $i\in\I$, $n\geq0$, $\lambda\in X$; see \cite[Lemma 23.2.2~(c)]{Lus93}. 

\begin{definition}[\text{cf. \cite[Definition 3.19]{BW18b}}]
\label{def:integral-Uidot}
The $\bar\A$-form $\dot{\U}^\imath_{\bar\A}$ is the set of elements $u\in\dot{\U}^\imath$, such that $u\cdot m\in \dot{\U}_{\bar\A}$ for all $m\in \dot{\U}_{\bar\A}$. Then $\dot{\U}^\imath_{\bar\A}$ is an $\bar\A$-subalgebra of $\dot{\U}^\imath$ which contains all the idempotents $\mathbf{1}_\zeta$ ($\zeta\in X_\imath$), and $\dot{\U}^\imath_{\bar\A}=\bigoplus_{\zeta\in X_\imath} \dot{\U}^\imath_{\bar\A}\mathbf{1}_\zeta$.
\end{definition}

\begin{lemma}[\text{cf. \cite[Lemma 3.20]{BW18b}}]
\label{lem:UidotA}
Let $u\in \dot{\U}^\imath$. Then $u\in \dot{\U}_{\bar\A}$ if and only if $u\cdot \mathbf{1}_\lambda\in \dot{\U}_{\bar\A}$ for each $\lambda\in X$.
\end{lemma}

By \cite[\S 37.1.3]{Lus93}, $T_j(K_\mu)=K_{s_j\mu}$ for $\mu\in Y^\imath,j\in \bI$. Using the definition~\eqref{def:Uilambda}, $T_j$ induces an automorphism on $\dot{\U}^\imath$ such that $T_j(\mathbf{1}_\zeta)=\mathbf{1}_{s_j\zeta}$ for $\zeta\in X_\imath$; cf. \cite[41.1.1]{Lus93}. 

\begin{proposition}
\label{prop:Tblack}
$\dot{\U}^\imath_{\bar{\A}}$ is invariant under the action of $T_j$ for $j\in \bI$.
\end{proposition}

\begin{proof}
Let $u\in \dot{\U}^\imath_{\bar{\A}}$. By definition, $u\cdot m\in \dot{\U}_{\bar{\A}}$ for any $m\in \dot{\U}_{\bar{\A}}$. By \cite[41.1.2]{Lus93}, we have $T_j(m)\in \dot{\U}_{\bar{\A}}$, and then $T_j(u)\cdot T_j(m)=T_j(u\cdot m)\in \dot{\U}_{\bar{\A}}$ for any $m\in \dot{\U}_{\bar{\A}}$. Since $m$ is arbitrary in $\dot{\U}_{\bar{\A}}$, $T_j(m)$ can be an arbitrary element in $\dot{\U}_{\bar{\A}}$, and hence  $T_j(u)\in \dot{\U}^\imath_{\bar{\A}}$ by the definition of $\dot{\U}^\imath_{\bar{\A}}$.
\end{proof}

In order to make braid group symmetries \cite{WZ23} and $\imath$-canonical bases \cite{BW18b} applicable, we shall also consider more general parameters, and consider $\Ui$ or $\dot{\U}^\imath$ over the field $\bF$. 

Following \cite[Table 3]{BW18b}, we define a balanced parameter $\bvs_\star=(\bvs_{i,\star})_{i\in\I_\circ}\in(\bF^\times)^{\I_\circ}$ such that the value of $\bvs_{i,\star}$ is given by the following table (recall $\bvs_{i,\star}$ depends only on real rank 1 Satake diagrams).

\begin{table}[h]
\caption{Value of $\bvs_{i,\star}$ ($i\in\I_\circ$) of real rank 1}
\label{table:dis-parameter}
\resizebox{5.5 in}{!}{%
\begin{tabular}{cccccccc}
\hline {AI}$_1$ & {AII}$_3$ & {AIII}$_{11}$ & \text { AIV, } {n} $\geq 2$ & \text { BII, } {n} $\geq 2$ & \text { CII, } {n} $\geq 3$ & \text { DII, } {n} $\geq 4$ & \text { FII } \\
\hline $ q^{-1}$ & $ q$ & $1$ & \eqref{eq:pstarA4} & $ q^{2n-3}$ & $ q^{n-1}$ & $ q^{n-2}$ & $ q^5$ \\
\hline & & & & & &
\end{tabular}
}
\end{table}
For (AIV), we define $\vs_{1,\star}=\vs_{n,\star}$ such that $\vs_{1,\star}^2=(-1)^nq^{n-1}$. i.e.,
\begin{align}
    \label{eq:pstarA4}
    \vs_{1,\star}=\vs_{n,\star}= (-q)^{n/2}q^{-1/2}.
\end{align}


\subsection{$\imath$-divided powers}
\label{sec:idiv}
We recall from \cite[Section 5]{BW21} a class of real rank 1 elements, called $\imath$-divided powers. We review the integrability for these elements.

\subsubsection{} 
\label{subsec:divid-1} Let $i\in\I_\circ$ be such that $\tau(i)\neq i$; $Y_i:=\vs_iT_{w_\bullet}(E_{\tau i})K_i^{-1}$. Set 
\begin{align}
\label{eq:div1}
B_i^{(m)}=&\frac{B_i^m}{[m]_i!}=\sum_{a=0}^m q_i^{-a(m-a)}Y_i^{(a)}F_i^{(m-a)}\in \tUi,\qquad \forall m\geq0,
\\
B_{i,\zeta}^{(m)}=&B_i^{(m)}\mathbf{1}_\zeta\in\dot{\U}^\imath,\qquad \forall \zeta\in X_\imath.
\end{align}


\subsubsection{}
\label{subsec:divid-2}
Let $i\in\I_\circ$ be such that $\tau(i)=i= w_\bullet i$; set 
\begin{align}
&& B_{i,\odd}^{(m)}=\frac{1}{[m]_{i}!}
\left\{ \begin{array}{ccccc}
B_i\prod_{j=1}^k \Big(B_i^2 - [2j-1]_{i}^2 q_i \vs_i  \Big) & \text{if }m=2k+1,\\
\prod_{j=1}^k \Big(B_i^2 - [2j-1]_{i}^2 q_i \vs_i  \Big) &\text{if }m=2k; \end{array}\right.
  \label{eq:iDPodd} \\
  \notag\\
&& B_{i,\ev}^{(m)}= \frac{1}{[m]_{i}!}
\left\{ \begin{array}{ccccc}
B_i\prod_{j=1}^k \Big(B_i^2 - [2j]_{i}^2 q_i \vs_i \Big) & \text{if }m=2k+1,\\
\prod_{j=1}^{k} \Big(B_i^2 - [2j-2]_{i}^2 q_i \vs_i \Big) &\text{if }m=2k. \end{array}\right.
 \label{eq:iDPev}
\end{align}


For $m\geq1$, $\zeta\in X_\imath$, define
\begin{align}
    B_{i,\zeta}^{(m)}:=\begin{cases}
        B_{i,\ev}^{(m)}\mathbf{1}_{\zeta}& \text{ if }\langle h_i, \zeta \rangle \in2\Z,
        \\
        B_{i,\odd}^{(m)}\mathbf{1}_{\zeta}& \text{ if }\langle h_i,\zeta \rangle \notin2\Z.
    \end{cases}
\end{align}   
For $\zeta\in X_\imath$ (which is given when we consider divided powers in $\dot{\U}^\imath$), we denote by $B_i^{(m)}\in\Ui$ (either $B_{i,\odd}^{(m)}$ or $B_{i,\ev}^{(m)}$, but unique) such that $B_{i,\zeta}^{(m)}=B_i^{(m)}\mathbf{1}_\zeta\in\dot{\U}^\imath$. 


\subsubsection{}
\label{subsec:divid-3}
Let $i\in\I_\circ$ be such that $\tau(i)=i\neq \bw i$; set
\begin{align}
\label{def:ZY}
    \mathfrak{Z}_i:=&\frac{\vs_ir_i(T_{\bw}(E_i))}{q_i^{-1}-q_i}\in\U_\bullet^+,\qquad
Y_i:=\vs_i T_{\bw}(E_i) K_i^{-1},
\\
\mathfrak{Z}_i^{(m)}:=&\frac{\mathfrak{Z}_i^m}{[m]_i!},\qquad Y_i^{(m)}=\frac{Y_i^{m}}{[m]_i!},\quad \forall m\geq0.
\end{align}


For any $m\geq0$, we define
\begin{align}
\label{eq:div3}
\begin{split}
    b_i^{(m)}:=&\sum_{a=0}^mq_i^{-a(m-a)} Y_i^{(a)}F_i^{(m-a)},
    \\ 
    B_i^{(m)}:=&b_i^{(m)}+\frac{q}{q-q^{-1}}\sum_{k\geq1}q_i^{k(k+1)/2}\mathfrak{Z}_i^{(k)}b_i^{(m-2k)}.
    \end{split}
\end{align}

\begin{lemma}
    For $m\geq0$, we have $B_i^{(m)}\in\Ui$.
\end{lemma}

\begin{proof}
    The proof is essentially the same as \cite[Lemma 5.9]{BW21}. 
\end{proof}

Define $B_{i,\zeta}^{(m)}=B_{i}^{(m)}\mathbf{1}_\zeta\in\dot{\U}^\imath$ for $\zeta\in X_\imath$.

\subsection{Properties for $\imath$-divided powers}
For all the $\imath$-divided powers $B_{i,\zeta}^{(m)}$ in the previous subsection, we have the following results.

\begin{proposition}
\label{prop:idiv-integral}
Let $\vs_i\in \bar{\A}$ for $i\in \wI$. We have $B_{i,\zeta}^{(m)}=B_{i}^{(m)}\mathbf{1}_\zeta\in\dot{\U}^\imath_{\bar{\A}}$ for any $m\geq0$, $\zeta\in X_\imath$, $i\in\I_\circ$.
\end{proposition}

\begin{proof}
   Recall that $\dot{\U}_{\bar{\A}}$ is generated by the elements $E_j^{(n)}\mathbf{1}_\lambda$, $F_j^{(n)}\mathbf{1}_\lambda$ for $j\in\I$, $\lambda\in X$. 
   By Lemma~\ref{lem:UidotA}, it is sufficient to show that $B_{i}^{(m)}\mathbf{1}_\lambda\in\dot{\U}_{\bar{\A}}$, for $i\in\I_\circ$, $m\geq0$. 
   
If $\tau(i)\neq i$, by \cite[\S5.5.1]{BW21}, we have 
\begin{align}
B_{i}^{(m)}\mathbf{1}_\lambda=\frac{B_i^{m}}{[m]_i!} \mathbf{1}_\lambda=\sum_{a=0}^mq_i^{-a(m-a)}Y_i^{(a)}F_i^{(m-a)}\mathbf{1}_\lambda,
\end{align}
where $Y_i=\vs_iT_{w_\bullet}(E_{\tau i})K_i^{-1}$. Note that $\big(T_{w_\bullet}(E_{\tau i})\big)^{(a)}=T_{w_\bullet}(E_{\tau i}^{(a)})\in\U^+_{\A}$, for any $a\in\N$; see \cite[Proposition 41.1.3]{Lus93}. Hence, $B_{i}^{(m)}\mathbf{1}_\lambda\in \dot{\U}_{\bar{\A}}$ in this case. 

If $\tau(i)=i\neq w_\bullet i$, using the definition \eqref{eq:div3}, we have $b_i^{(m)}\mathbf{1}_\lambda\in \dot{\U}_{\bar{\A}}$ similar to the above case. Hence, by definition \eqref{eq:div3}, it is enough to prove that $\frac{\mathfrak{Z}_i^{(m)}}{q-q^{-1}}\in\U_{\bar{\A}}^+$. By \cite[Proposition 5.6]{BW21}, we have $\mathfrak{Z}_i^{(m)}\in\U_{\bar{\A}}^+$ for $m\geq0$. Consider $\U$ and $\Ui$ over $\Q(v)$. Then there exists a bar involution $\psi$ of the $\Q$-algebra $\U$ such that $v\mapsto v^{-1}$, $E_j\mapsto E_j$, $F_j\mapsto F_j$, $K_j\mapsto K_j^{-1}$ for $j\in\I$; cf. \cite[\S3.1.12]{Lus93}. By using $\psi$ and the proof of \cite[Lemma 5.7]{BW21}, we know 
$\frac{r_i(T_{\bw}(E_i))}{(q_i^{-1}-q_i)(q-q^{-1})}\in\U^+_\A$, and then 
$\mathfrak{Z}_i/(q-q^{-1})\in \U^+_{\bar{\A}}$. Therefore,  $$\frac{\mathfrak{Z}_i^{(m)}}{q-q^{-1}}=\frac{\mathfrak{Z}_i^m}{[m]_i!(q-q^{-1})}\in\U_{\bar{\A}}^+,$$  by using \cite[Lemma 5.11]{BW21}.

It remains to consider the case $\tau(i)=i= w_\bullet i$. For $\vs_i=q_i^{-1}$, the formulas of $B_{i}^{(m)}\mathbf{1}_\lambda$ in terms of elements in $\dot{\U}$ are given in \cite[Propositions 2.8, 3.5]{BeW18}. By similar computations {\em loc. cit.}, we can generalize Berman-Wang's formulas to general parameter $\vs_i$; for example,
if $\langle h_i,\lambda\rangle=2d$, for $d\in\Z$, we have
\begin{align*}
B_{i}^{(2n)}\mathbf{1}_\lambda=B_{i,\ev}^{(2n)}\mathbf{1}_\lambda
&\notag
= \sum_{c=0}^n \sum_{a=0}^{2n-2c} (q_i\vs_i)^{a+c} q_i^{2(a+c)(n-a-d)-2ac-\binom{2c+1}{2}}  \qbinom{n-c-a-d}{c}_{q_i^2}\times
\\ &\qquad\qquad \times E^{(a)}_i  F^{(2n-2c-a)}_i\mathbf{1}_\lambda.\label{t2mdot}
\end{align*}
Using these formulas, we conclude $B_{i,\zeta}^{(m)}\in \dot{\U}^\imath_{\bar{\A}}$ in this case.  
\end{proof}

Let $\bvs=(\vs_i)_{i\in\I_\circ}$ and $\bvs'=(\vs_i')_{i\in\I_\circ}$ be two balanced parameters. 
Similar to Lemma~\ref{lem:iso-paremeter} (see also \cite[Lemma 2.5.1]{Wat21}), there exists an isomorphism \begin{align}
\phi_{\bvs,\bvs'}:=\phi_{\bvs'}\circ\phi_{\bvs}^{-1}:\Ui_{\bvs}\rightarrow \Ui_{\bvs'}
\end{align}
which sends 
$B_i\mapsto \sqrt{\vs_{i}(\vs_i')^{-1} }B_i$, $E_j\mapsto E_j$, $K_j\mapsto K_j$, $k_r\mapsto k_r$, for $i\in \I_\circ$, $j\in \I_\bullet$, $r\in\I\setminus\I_{\circ,\tau}$. Then $\phi_{\bvs,\bvs'}$ preserves the ideal in the definition \eqref{def:Uilambda} of ${}_{\lambda'}\Ui_{\lambda''}$. Hence, $\phi_{\bvs,\bvs'}$ induces an isomorphism (also denoted by $\phi_{\bvs,\bvs'}$)
\begin{align}
\phi_{\bvs,\bvs'}:\dot{\U}^{\imath}_{\bvs}\rightarrow \dot{\U}^\imath_{\bvs'},
\end{align}
such that $\phi_{\bvs,\bvs'}(\mathbf{1}_\lambda)=\mathbf{1}_\lambda$. The inverse of $\phi_{\bvs,\bvs'}$ is given by $\phi_{\bvs',\bvs}$.

\begin{lemma}
\label{lem:div-iso}
Keep notations as above. Then the isomorphism $\phi_{\bvs,\bvs'}:\dot{\U}^{\imath}_{\bvs}\rightarrow \dot{\U}^\imath_{\bvs'}$ satisfies 
\begin{align}
    \phi_{\bvs,\bvs'}(B_{i,\zeta}^{(m)})=\sqrt{\vs_i^m(\vs_i')^{-m}}B_{i,\zeta}^{(m)},\qquad \forall i\in\I_\circ,\;\zeta\in X_\imath\;m\geq0.
\end{align}
\end{lemma}

\begin{proof}
We only prove the statement for the case $\tau(i)=i\neq w_\bullet i$, since the other two cases are similar. Set $a_i:= (\vs_i)^{-1}\vs_i'$ for $i\in \wI$ and $a_i=1$ for $i\in\I_\bullet$. Let $\Phi_{\ba}: \U_\bF\rightarrow \U_\bF$ be the map defined in \eqref{def:Phi} for $\ba=(a_i)_{i\in\I}$. 
It is clear from the definition that $\phi_{\bvs,\bvs'}$ equals the restriction of $\Phi_{\ba}$.
By \eqref{def:Phi} and \eqref{def:ZY}, we have $\Phi_{\ba}(\mathfrak{Z}_i)=a_i^{-1}\mathfrak{Z}_i$, $\Phi_{\ba}(Y_i)=a_i^{-1/2} Y_i$. Then by \eqref{eq:div3}, $\Phi_{\ba}(b_i^{(m)})=a_i^{-m/2} b_i^{(m)}$, $\Phi_{\ba}(B_i^{(m)})=a_i^{-m/2} B_i^{(m)}$. Therefore,  $\phi_{\bvs,\bvs'}(B_i^{(m)})=a_i^{-m/2} B_i^{(m)}$ as desired.
\end{proof}

\begin{proposition}
\label{prop:integralformsequiv}
Let $\vs_i=\pm v^{a_i}$ for some $a_i\in\Z$ for any $i\in\I_\circ$.
The $\bar{\A}$-form $\dot{\U}^\imath_{\bar{\A}}$ is generated  by the $\imath$-divided powers $B_{i,\zeta}^{(n)}$ ($i\in\I_\circ$) and $E_j^{(n)}\mathbf{1}_\zeta$, $F_j^{(n)}\mathbf{1}_\zeta$ ($j\in\I_\bullet$), for $n\geq1$ and $\zeta\in X_\imath$.
\end{proposition}

 \begin{proof}
 In order to avoid confusions, we denote  $\dot{\U}^\imath_{\bar{\A}}$ by $\dot{\U}^\imath_{\bvs,\bar{\A}}$. 
 
 We first consider the case $\bvs=\bvs_\star=(\vs_{i,\star})_{i\in\I_\circ}$. Since $\vs_{i,\star}\in \{\pm q^{1/2\Z},\pm (-q)^{1/2\Z}\}$ for $i\in\I_\circ$, we introduce a ring $\widetilde{\A}=\Z[v^{\pm1/2},(-v)^{\pm1/2}]$. Then, using the same arguments of \cite{BW18b,BW21}, there exists an $\imath$-canonical basis
 on the $\widetilde{\A}$-form $\dot{\U}^\imath_{\bvs_\star,\widetilde{\A}}$, and this shows that $\dot{\U}^\imath_{\bvs_\star,\widetilde{\A}}$ is generated  by the $\imath$-divided powers $B_{i,\zeta}^{(n)}$ ($i\in\I_\circ$) and $E_j^{(n)}\mathbf{1}_\zeta$, $F_j^{(n)}\mathbf{1}_\zeta$ ($j\in\I_\bullet$), for $n\geq1$ and $\zeta\in X_\imath$; see \cite[Corollary 7.5]{BW21}. Moreover, $\dot{\U}^\imath_{\bvs_\star,\widetilde{\A}}$ is a free $\widetilde{\A}$-module.

We next consider a general parameter $\bvs$ satisfying $\vs_i=\pm v^{a_i}, a_i\in \Z$. By Lemma~\ref{lem:div-iso}, the isomorphism $\phi_{\bvs,\bvs_\star}:\dot{\U}^{\imath}_{\bvs}\rightarrow \dot{\U}^\imath_{\bvs_\star}$ induces an isomorphism $\phi_{\bvs,\bvs_\star}:\dot{\U}^\imath_{\bvs,\widetilde{\A}}\rightarrow \dot{\U}^\imath_{\bvs_\star,\widetilde{\A}} $ of $\widetilde{\A}$-algebras by noting that
$\sqrt{\vs_i(\vs_{i,\star})^{-1}}\in \{v^{\pm1/2\Z},(-v)^{\pm1/2\Z}\}$ which is invertible in $\widetilde{\A}$. It follows from Lemma \ref{lem:div-iso} that  $\dot{\U}^\imath_{\bvs,\widetilde{\A}}$  is generated  by the $\imath$-divided powers $B_{i,\zeta}^{(n)}$ ($i\in\I_\circ$) and $E_j^{(n)}\mathbf{1}_\zeta$, $F_j^{(n)}\mathbf{1}_\zeta$ ($j\in\I_\bullet$), for $n\geq1$ and $\zeta\in X_\imath$.  

We have $\bar{\A}\subset\widetilde{\A}$ is an extension of rings with $\widetilde{\A}$ projective (and then flat) as $\bar{\A}$-module since $\widetilde{\A}=\bar{\A}\oplus \bar{\A}v^{1/2}\oplus \bar{\A}(-v)^{1/2}\cong \bar{\A}^{\oplus 3}$ as $\bar{\A}$-module. So the natural homomorphism (base change) $\dot{\U}^\imath_{\bvs,\bar{\A}}\rightarrow \dot{\U}^\imath_{\bvs,\widetilde{\A}}$ is injective; see e.g. \cite[Theorem 5.3.13]{WK17}. 

On the other hand, note that $\Q(v)\cap \Z[v^{\pm1/2},(-v)^{\pm1/2}]=\Z[v,v^{-1}]=\bar{\A}$. The $\Q(v)$-algebra $\tUi_{\bvs,\Q(v)}$ is generated by  the $\imath$-divided powers $B_{i,\zeta}^{(n)}$ ($i\in\I_\circ$) and $E_j^{(n)}\mathbf{1}_\zeta$, $F_j^{(n)}\mathbf{1}_\zeta$ ($j\in\I_\bullet$), for $n\geq1$ and $\zeta\in X_\imath$. Note that $\tUi_{\bvs,\Q(v)}\cap \tUi_{\bvs,\widetilde{\A}}=\tUi_{\bvs,\bar{\A}}$ since $\tUi_{\bvs,\widetilde{\A}}$ is a free $\widetilde{\A}$-module. 
Hence, $\dot{\U}^\imath_{\bvs,\bar{\A}}$ is generated by the $\imath$-divided powers $B_{i,\zeta}^{(n)}$ ($i\in\I_\circ$) and $E_j^{(n)}\mathbf{1}_\zeta$, $F_j^{(n)}\mathbf{1}_\zeta$ ($j\in\I_\bullet$), for $n\geq1$ and $\zeta\in X_\imath$. 
 \end{proof}



\subsection{Divided powers of root vectors for rank 1 Satake diagrams}
\label{subsec:divid-root-rank1}

In this subsection, we shall construct $\imath$divided powers of root vectors for rank 1 Satake diagrams in Table \ref{table:localSatake}. We also assume $\vs_i\in\bar{\A}=\Z[v,v^{-1}]$. 

\subsubsection{$\un{\text{Type  AI}_1}$}
In this case, 
$\cR^+(\bs_1)=\alpha_1$. Set $B_{\alpha_1}=B_1\in\tUi$. Define $B_{\alpha_1}^{(m)}$ to be $B_{i}^{(m)}$ as in \S\ref{subsec:divid-1}, for any $m\geq0$, since we shall consider it in $\dot{\U}^\imath$.  

\subsubsection{$\un{\text{Type  AII}_3}$}
In this case, 
$\cR^+(\bs_2)=\{\alpha_2,\alpha_1+\alpha_2,\alpha_2+\alpha_3,\alpha_1+\alpha_2+\alpha_3\}$. 
For any $m\geq0$, define $B_{\alpha_2}^{(m)}$ to be the $\imath$-divided power $B_{2}^{(m)}$ given by \eqref{eq:div3}; $B_{\alpha_1+\alpha_2}^{(m)}=\tTD_1^{-1}(B_{\alpha_2}^{(m)})$, $B_{\alpha_2+\alpha_3}^{(m)}=\tTD_3^{-1}(B_{\alpha_2}^{(m)})$, $B_{\alpha_1+\alpha_2+\alpha_3}^{(m)}=\tTD_{13}^{-1}(B_{\alpha_2}^{(m)})$. 

\subsubsection{$\un{\text{Type  AIII}_{11}}$} 

We have $\cR^+(\bs_1)=\{\alpha_1,\alpha_2\}$. Let $B_{\alpha_1}^{(m)}=\frac{B_1^m}{[m]^!}$, $B_{\alpha_2}^{(m)}=\frac{B_2^m}{[m]^!}$ for any $m\geq0$.

\subsubsection{$\un{\text{Type  AIV},n\geq2}$ }

In this case, $\cR^+(\bs_1)=\{\beta_i\mid 1\leq i\leq 2n-1 \}$.
We define for $j\neq n$, $m\geq0$
\begin{align*}
B_{\beta_j}^{(m)}=\frac{B_{\beta_j}^m}{[m]!}
=
\begin{cases}
\tTD_{2\cdots j}^{-1}(B_1^{(m)}),&\qquad 1\leq j\leq n-1,
\\
\tTD_{(n-1)\cdots(2n-j+1) }^{-1} (B_n^{(m)}),&\qquad n+1\leq j\leq 2n-1,
\end{cases}
\end{align*}
as Lusztig's divided powers.

For $j=n$, note that 
\begin{align*}
F_{\beta_n}^{(m)}=&\tTD_{1\cdots(n-1)}(F_n^{(m)})=\sum_{r=0}^{m}(-1)^rq^{r}\tTD_{2\cdots(n-1)}^{-1}(F_{1}^{(r)})(F_n^{(m)})\tTD_{2\cdots(n-1)}^{-1}(F_{1}^{(m-r)})
\\
=&\sum_{r=0}^{m}(-1)^rq^{r}F_{\beta_{n-1}}^{(r)}F_n^{(m)}F_{\beta_{n-1}}^{(m-r)}.
\end{align*}
We define 
\begin{align*}
B_{\beta_n}^{(m)}=&\sum_{r=0}^{m}(-1)^rq^{r}\tTD_{2\cdots(n-1)}^{-1}(B_{1}^{(r)})(B_n^{(m)})\tTD_{2\cdots(n-1)}^{-1}(B_{1}^{(m-r)})
\\
=&\sum_{r=0}^{m}(-1)^rq^{r}B_{\beta_{n-1}}^{(r)}B_n^{(m)}B_{\beta_{n-1}}^{(m-r)},\qquad \text{ for } m\geq0.
\end{align*}
Then the leading term of $B_{\beta_n}^{(m)}$ is  $F_{\beta_n}^{(m)}$.


\subsubsection{$\un{\text{Type BII}, n\geq2}$} 

In this case, $\cR^+(\bs_1)=\{\beta_i\mid 1\leq i\leq 2n-1\}$.
For $m\geq0$, we define 
\begin{align*}
    B_{\beta_j}^{(m)}=\begin{cases}   
    \tTD_{2\cdots i}^{-1}(B_1^{(m)}), &\text{ if } 1\leq j\leq n-1,
    \\
    \\
    \sum\limits_{r=0}^{m}(-1)^rq^{r}B_{\beta_{n-1}}^{(r)}F_n^{(m)}B_{\beta_{n-1}}^{(m-r)},
    &\text{ if }j= n,
    \\
    \\
    \tTD_{2\cdots n\cdots(2n-j+1)}^{-1}(B_1^{(m)}),
    & \text{ if } n+1\leq j\leq 2n-1.
    \end{cases}
\end{align*}
where $B_1^{(m)}$ are the $\imath$-divided powers defined by \eqref{eq:div3}. Note that $B_{\beta_n}^{(m)}$ is defined similar to Type AIV, and its leading term is $F_{\beta}^{(m)}$. The same argument also holds for the following cases, we do not repeat it again.  

\subsubsection{$\un{\text{Type  CII},n\geq3}$ }

$\cR^+(\bs_2)=\{\beta_i\mid 1\leq i\leq 4n-5\}$.
For $j\neq n$, $m\geq0$, we define 
\begin{align*}
B_{\beta_j}^{(m)}=\begin{cases}  
     \sum_{r=0}^{2m}(-1)^rq_{n-1}^{r}B_{\beta_{n-2}}^{(r)}B_n^{(m)}B_{\beta_{n-2}}^{(m-r)}
     &\text{ if }j=n-1,
     \\
     \\
     \sum_{r=0}^{m}\sum_{s=0}^{m}(-1)^{r+s}q_{2}^{r+s}B_{\beta_{2n-3}}^{(r)}B_2^{(s)}B_1^{(m)}B_2^{(m-s)}B_{\beta_{2n-3}}^{(m-r)}&\text{ if }j=2n-2,
     \\
     \\
     \sum_{r=0}^{2m}(-1)^rq_{n-1}^{r}B_{\beta_{3n-2}}^{(r)}B_n^{(m)}B_{\beta_{3n-2}}^{(m-r)}
     &\text{ if }j=3n-3,
     \\
     \\
    \frac{B_{\beta_j}^m}{[m]!} &\text{ otherwise}.
    \end{cases}
\end{align*}
\subsubsection{$\un{\text{Type  DII},n\geq4}$ }
In this case, $\cR^+(\bs_1)=\{\beta_i\mid 1\leq i\leq 2n-2\}$. 
Let $1\leq j\leq 2n-2$. Note that $B_{\beta_j}$ is of the form $\tTD_w^{-1}(B_1)$ for some $w\in \bW$; we then define $B_{\beta_j}^{(m)}=\tTD_w^{-1}(B_1^{(m)})$ for $m\geq 0$ where $B_1^{(m)}$ are the $\imath$-divided powers given by \eqref{eq:div3}.

\subsubsection{$\un{\text{Type  }F_4}$ }

$\cR^+(\bs_4)=\{\beta_i\mid 1\leq i\leq 15\}$.
For $j\neq n$, $m\geq0$, we define 
\begin{align*}
    B_{\beta_j}^{(m)}=\begin{cases}   \frac{B_{\beta_j}^m}{[m]!} &\text{ if }j=1,2,4,7,9,11,14,15,
    \\
    \\
    \sum_{r=0}^{2m}(-1)^rq_2^{r}B_{\beta_{2}}^{(r)}B_2^{(m)}B_{\beta_{2}}^{(m-r)}
    &\text{ if }j=3,
     \\
     \\
    \sum_{r=0}^{m}(-1)^rq_3^{r}B_{\beta_{3}}^{(r)}B_1^{(m)}B_{\beta_{3}}^{(m-r)}
    &\text{ if }j=5,
    \\
    \\
     \sum_{r=0}^{m}(-1)^rq_1^{r}B_{\beta_{5}}^{(r)}B_2^{(m)}B_{\beta_{5}}^{(m-r)}
    &\text{ if }j=6,
    \\
    \\
   {\sum_{r=0}^{m}(-1)^rq_3^{r}B_{\beta_{7}}^{(r)}\tTD_{323}^{-1}(B_4^{(m)})B_{\beta_{7}}^{(m-r)}} &\text{ if }j=8,
     \\ 
     \\
     \sum_{r=0}^{2m}(-1)^rq_3^{r}B_{\beta_{9}}^{(r)}B_1^{(m)}B_{\beta_{9}}^{(m-r)}
     &\text{ if }j=10,
     \\
     \\
     \sum_{r=0}^{m}(-1)^rq_2^{r}B_{\beta_{10}}^{(r)}B_2^{(m)}B_{\beta_{10}}^{(m-r)}
     &\text{ if }j=12,
     \\
     \\
     \sum_{r=0}^{m}(-1)^rq_1^{r}B_{\beta_{12}}^{(r)}B_1^{(m)}B_{\beta_{12}}^{(m-r)}
     &\text{ if }j=13.
    \end{cases}
\end{align*}

In conclusion, we have the following result from the construction of divided powers of root vectors.
\begin{lemma}
\label{lem:leading-div}
For the rank 1 Satake diagram in Table \ref{table:localSatake}, we have the leading term of $B_{\beta}^{(m)}$ is $F_\beta^{(m)}$, for any $\beta\in\cR^+(\bs_i)$, and $m\geq0$.
\end{lemma}

Let $\cR^+_\bullet=\{\gamma_1,\dots,\gamma_t\}$. 
Using a reduced expression of the longest element $w_\bullet$ of $W_\bullet$, we can define the root vectors $E_{\gamma_j}$ and $F_{\gamma_j}$ for $1\leq j\leq t$,  by using the braid group symmetries $T_i$ ($i\in\I_\bullet$).
We can also define $E_{\gamma_j}^{(m)}$ and $F_{\gamma_j}^{(m)}$ for $m\geq0$; see Proposition \ref{prop:PBW-QG}. 

\begin{lemma}
\label{lem:div-integral}
Let $\bvs=(\vs_i)_{i\in\I_\circ}$ with $\vs_i\in\bar{\A}$. 
For the rank 1 Satake diagram in Table \ref{table:localSatake}, we have $B_{\beta_j}^{(m)}\mathbf{1}_\zeta\in \dot{\U}^\imath_{\bar{\A}}$, for any $\beta_j\in \cR^+(\bs_i)=\{\beta_1,\dots,\beta_r\}$ and $\zeta\in X_\imath$. 
\end{lemma}

\begin{proof}
By our construction, each $B_{\beta_j}^{(m)}$ is an $\bar{\A}$-linear combination of products of elements with the form 
\[
T_w^{-1}(B_{i }^{(r)}),\qquad  w\in W_\bullet,\; 0\leq r \leq m.
\]
($B_{i }^{(r)}$ are the $\imath$-divided powers given in \S \ref{sec:idiv}.) Then the desired statement follows from Proposition~\ref{prop:Tblack} and Proposition~\ref{prop:idiv-integral}. 
\end{proof}

\section{PBW bases of modified $\imath$quantum groups}
\label{subsec:braid-Uidot}

In this section, we shall construct a PBW basis for the integral form $\dot{\U}^\imath_{\bar{\A}}$ with some restrictions on the parameters $\bvs$.

\subsection{Distinguished parameter case}

We establish the integrality for the quasi $K$-matrix and relative braid group symmetries for the distinguished parameter $\bvs_\diamond$ in this subsection.

Recall that $\bar{\A}=\Z[q^{1/2},q^{-1/2}]$. 
Following \cite[Section 9.4]{WZ23}, the symmetry $\TT_i:\tUi\rightarrow\tUi$ induces an automorphism $\bT_{i,\bvs_\diamond}: \Ui_{\bvs_\diamond}\rightarrow \Ui_{\bvs_\diamond}$ such that the following diagram commutes:
\begin{align}
\label{eq:Ti-distin}
    \xymatrix{\tUi\ar[rr]^{\TT_i} \ar[d]^{\pi^\imath_{\bvs_\diamond}}&& \tUi \ar[d]^{\pi^\imath_{\bvs_\diamond}}\\
\Ui_{\bvs_\diamond}\ar[rr]^{\bT_{i,\bvs_\diamond}} &&\Ui_{\bvs_\diamond}}
\end{align}

By \cite[Proposition 4.11]{WZ23} and the definition of $\pi^\imath_{\bvs_\diamond}$, we have $\bT_{i,\bvs_\diamond}(K_\mu)=K_{\bs_i\mu}$ for $\mu\in Y^\imath$. Then the symmetry $\bT_{i,\bvs_\diamond}: \Ui_{\bvs_\diamond}\rightarrow \Ui_{\bvs_\diamond}$ induces for each $\lambda',\lambda''$ a linear isomorphism
${}_{\lambda'}\Ui_{\bvs_\diamond,\lambda''}\rightarrow {}_{\bs_i\lambda'}\Ui_{\bvs_\diamond,\bs_i\lambda''}$, and then an algebra automorphism $\bT_{i,\bvs_\diamond}:\dot{\U}^\imath_{\bvs_\diamond}\rightarrow \dot{\U}^\imath_{\bvs_\diamond}$ such that $\bT_{i,\bvs_\diamond}(\mathbf{1}_\lambda)=\mathbf{1}_{\bs_i\lambda}$ for all $\lambda\in X_\imath$ and $$\bT_{i,\bvs_\diamond}(uxx'u')=\bT_{i,\bvs_\diamond}(u)\bT_{i,\bvs_\diamond}(x)\bT_{i,\bvs_\diamond}(x')\bT_{i,\bvs_\diamond}(u'),$$ 
for all $u,u'\in\Ui_{\bvs_\diamond},x,x'\in\dot{\U}^\imath_{\bvs_\diamond}$. 

Let $M$ be a finite-dimensional $\U$-module of type $\mathbf{1}$. By \cite{Lus93}, there are linear operators $T_i$ on $M$ for all $i\in\I $, such that 
\begin{align}
    T_{i}(uv)=T_i(u)T_i(v),\qquad u\in \U,v\in M. 
\end{align}

Regard the $\U$-module $M$ as a $\Ui_{\bvs_\diamond}$-module by restriction.
By \cite[Theorem 10.5]{WZ23}, for any $i\in\I_\circ$, we have linear operators $\bT_{i,\bvs_\diamond}(m):=T_i(\Upsilon_{i,\bvs_\diamond}^{-1})T_{\bs_i}(m)$
for any $m\in M$, and 
\begin{align}
\label{braid-imod}
    \bT_{i,\bvs_\diamond}(xm)=\bT_{i,\bvs_\diamond}(x)\bT_{i,\bvs_\diamond}(m),\qquad x\in \Ui_{\bvs_\diamond},m\in M. 
\end{align}
The inverse of $\bT_{i,\bvs_\diamond}$ is given by 
\begin{align}
\label{eq:Timod-1}
\bT_{i,\bvs_\diamond}^{-1}(m):=\Upsilon_{i,\bvs_\diamond} T_{\bs_i}^{-1}(m),
\end{align}
for any $m\in M$. 

Let $M$ be an $X_\imath$-weight module for $\Ui$. Since any weight $\Ui$-module can naturally be regarded as a $\dot{\U}^\imath$-module, \eqref{braid-imod} also holds for $\dot{\U}^\imath_{\bvs_\diamond}$.

\begin{lemma}
\label{lem:k-diamond}
For the parameter $\bvs_\diamond$, the quasi $K$-matrix $\fX_{i,\bvs_\diamond}=\sum_{\mu\in \N\I}\fX_{i,\bvs_\diamond}^\mu$ satisfies $\Upsilon_{i,\bvs_\diamond}^\mu\in \U^+_{\bar{\A}}$ for all $\mu$.
\end{lemma}

\begin{proof}

If the underlying symmetric pair is of type AIV, then the following formula for the quasi $K$-matrix $\fX_{i,\bvs}$
\begin{align}\label{eq:DK}
\fX_{i,\bvs}=\left(\sum_{k\geq 0} (-\vs_1)^k q^{-k(k-1)/2} T_{s_1\bw}(E_n)^{(k)}\right)\left(\sum_{k\geq 0} (-\vs_n)^k q^{-k(k-1)/2} T_{s_n\bw}(E_1)^{(k)}\right)
\end{align}
is obtained in \cite[Lemma 3.10]{DK19} when the parameter $\bvs=(\vs_1,\vs_n)$ satisfy the condition $\vs_1 = (-1)^n q^{n-1}\ov{\vs_n}$. Using \cite[Proposition~3.3]{Ko21}, one can show that \eqref{eq:DK} is also valid for the distinguished parameter $\bvs_{\diamond}$ (even $\bvs_{\diamond}$ does not satisfy that condition). The desired statement is then clear from the formula \eqref{eq:DK}, since $\vs_{i,
\diamond}\in \bar{\A}$.

It suffices to prove the statement for all rank 1 Satake diagrams other than type AIV. Let $(\I=\{i\}\cup\bI,\Id)$ be such a Satake diagram. Denote by $\Ui_{\bvs_\star}$ the $\imath$quantum group with the parameter $ \vs_{i,\star}$ given in Table~\ref{table:dis-parameter}. Note that $\vs_{i,\star}\in \Z[q,q^{-1}]$. Using the arguments in \cite[Appendix A]{BW18b}, one can prove that $\Upsilon_{i,\bvs_\star}= \sum_{\mu\in \N\I}\Upsilon_{i,\bvs_\star}^\mu$ satisfies $\Upsilon_{i,\bvs_\star}^\mu\in \U^+_{\bar{\A}}$ for all $\mu$.

Recall the isomorphism  $\phi_{\bvs_\star}:\U^{\imath}_{\bvs_\diamond}\rightarrow \U^\imath_{\bvs_\star}$ by viewing them as $\bF$-algebras in Lemma \ref{lem:iso-paremeter}. Define $\ba=(a_i)_{i\in \I}$ by setting $a_i:= (\vs_{i,\diamond})^{-1}\vs_{i,\star}$ and $a_i=1$ for $i\in\I_\bullet$. Let $\Phi_{\ba}: \U_\bF\rightarrow \U_\bF$ be the map defined in \eqref{def:Phi} for $\ba=(a_i)_{i\in\I}$. It is clear that $\phi_{\bvs_\star}$ is the restriction of $\Phi_{\ba}$. By the uniqueness of quasi K-matrices, we know 
\begin{align}\label{eq:PhifX}
\fX_{i,\bvs_\diamond} =\Phi_\ba^{-1}(\fX_{i,\bvs_\star}),
\qquad \fX_{i,\bvs_\diamond}^\mu=\Phi_\ba^{-1}(\fX_{i,\bvs_\star}^\mu),
\end{align}
for all $\mu$. 
Write $a_{\mu}=\prod_{i} a_i^{k_i}$ for $\mu=\sum_{i\in \I} k_i\alpha_i$. By \eqref{def:Phi}, $\Phi_\ba$ restricts to the weight space $\tU^+_\mu$ and $\Phi_\ba(x)=a_{\mu}^{1/2}x$ for any $x\in \tU^+_\mu$.

Let $\mu\in\N\I$ such that $\Upsilon_{i,\bvs_\star}^\mu\neq0$. By \cite[Lemma 3.6]{DK19}, $\mu=k(\alpha_i-\theta\alpha_i)$ for some $k\in \N$, and then $\mu$ has the form $2k\alpha_i+\mu_\bullet$ for some $k\in \N, \mu_\bullet\in \N\bI$. It follows that $\Phi_\ba(x)=a_{i} x$ for any $x\in \tU^+_\mu$. Since $a_i\in \bar{A}$ according to Table~\ref{table:localSatake} and Table~\ref{table:dis-parameter}, using \eqref{eq:PhifX}, we have $\Upsilon_{i,\bvs_\diamond}^\mu\in \U^+_{\bar{\A}}$.


\end{proof}

\begin{proposition}
\label{lem:integral-braid}
For the parameter $\bvs_\diamond$, the $\bar{\A}$-subalgebra $\dot{\U}^\imath_{\bar{\A}}$ is invariant under the actions of $\bT_{i,\bvs_\diamond}$, for $i\in\wItau$.
\end{proposition}

\begin{proof}
By \cite[Corollary 3.21]{BW18b}, it is enough to prove that $\bT_{i,\bvs_\diamond}(u)\big(_{\bar{\A}} L(\lambda)\big)\subseteq {}_{\bar{\A}} L(\lambda)$ for all $\lambda\in X^+$, where ${}_{\bar{\A}}L(\lambda)=\U^{-}_{\bar{\A}}\eta_\lambda$.

From the formulas of $T_i$, $T_i^{-1}$ in \cite[\S 5.2.1]{Lus93}, we know that $T_i({}_\A L(\lambda))\subseteq {}_\A L(\lambda)$, $T_i^{-1}({}_\A L(\lambda))\subseteq {}_\A L(\lambda)$ and then $T_i({}_{\bar{\A}} L(\lambda))\subseteq {}_{\bar{\A}} L(\lambda)$, $T_i^{-1}({}_{\bar{\A}} L(\lambda))\subseteq {}_{\bar{\A}} L(\lambda)$. By Lemma~\ref{lem:k-diamond}, $\Upsilon_{i,\bvs_\diamond}$ and $\Upsilon_{i,\bvs_\diamond}^{-1}$ are in the completion of $\U^+_{\bar\A}$. Using \eqref{eq:Timod-1}, we conclude that 
$\bT_{i,\bvs_\diamond}(m)$ and $\bT_{i,\bvs_\diamond}^{-1}(m)$ are in ${}_{\bar{\A}} L(\lambda)$ for any $m\in {}_{\bar{\A}} L(\lambda)$. 
The proof is completed. 
\end{proof}

\subsection{General balanced parameter case}

Let $\bvs$ be a balanced parameter. There exist automorphisms $\bT_{i,\bvs}$ on $\Ui_\bvs$ such that the following diagram commutes:
\begin{align}
\label{eq:Ti-general}
\xymatrix{
\Ui_{\bvs_\diamond} \ar[rr]^{\bT_{i,\bvs_\diamond}} \ar[d]^{\phi_\bvs} && \Ui_{\bvs_\diamond}\ar[d]^{\phi_\bvs}
\\
\Ui_{\bvs}\ar[rr]^{\bT_{i,\bvs}} &&\Ui_\bvs
}
\end{align}
By Lemma~\ref{lem:iso-paremeter}, $\phi_\bvs(K_\mu)=K_\mu$ and hence $\bT_{i,\bvs}(K_\mu)=\phi_{\bvs}\bT_{i,\bvs_\diamond}\phi_{\bvs}^{-1}(K_\mu)=K_{\bs_i \mu}$. Then the symmetry $\bT_{i,\bvs}$ induces an algebra automorphism $\bT_{i,\bvs}:\dot{\U}^\imath_\bvs \rightarrow \dot{\U}^\imath_\bvs$ such that $\dot{\U}^\imath_\bvs(\mathbf{1}_\zeta)=\mathbf{1}_{\bs_i\zeta}$ for all $\zeta\in X_\imath$. By \eqref{eq:Ti-general}, we have the following commutative diagram
\begin{align}
\label{eq:Ti-dot}
\xymatrix{
\dot{\U}^\imath_{\bvs_\diamond} \ar[rr]^{\bT_{i,\bvs_\diamond}} \ar[d]^{\phi_\bvs} && \dot{\U}^\imath_{\bvs_\diamond}\ar[d]^{\phi_\bvs}
\\
\dot{\U}^\imath_{\bvs}\ar[rr]^{\bT_{i,\bvs}} &&\dot{\U}^\imath_\bvs
}
\end{align}

\begin{theorem}
\label{prop:integral-braid}
 Let $\bvs=(\vs_i)_{i\in\I_\circ}$ be parameters such that $\vs_i$ has the form $ \vs_i =q^{a_i}\vs_{\diamond,i},a_i\in\Z$ for any $i\in\I_\circ$. Then the $\bar{\A}$-form  $\dot{\U}^\imath_{\bar{\A}}$ is invariant under the actions of $\bT_{i,\bvs}$, for $i\in\wItau$. 
\end{theorem}

\begin{proof}
Note that $\vs_i\in\bar{\A}$ for $i\in\I_\circ$ by our assumption. 
Denote $\bT_{i,\bvs}$ by $\bT_i$ in the following. By Lemma~\ref{lem:div-iso} and Proposition~\ref{prop:integralformsequiv}, the map $\phi_{\bvs_\diamond,\bvs}$ induces an isomorphism of $\bar{\A}$-forms $\phi_{\bvs_\diamond,\bvs}:\dot{\U}^\imath_{\bvs_\diamond,\bar{\A}}\rightarrow \dot{\U}^\imath_{\bvs,\bar{\A}}=\dot{\U}^\imath_{\bar{\A}}$. 
Together with Proposition~\ref{lem:integral-braid} and \eqref{eq:Ti-dot}, the $\A$-form  $\dot{\U}^\imath_{\bar{\A}}$ is invariant under the actions of $\bT_{i}$.  
\end{proof}

Recall the reduced expressions $\bw=s_{j_1}s_{j_2}\cdots s_{j_r}$ and  $\bbw_0=s_{k_1}s_{k_2}\cdots s_{k_{t}}$ in $W$, obtained using the reduced expressions for $\bs_i$. We obtain a reduced expression of $w_0$ by composing $\bbw_0=s_{k_1}s_{k_2}\cdots s_{k_{t}}$ and $\bw=s_{j_1}s_{j_2}\cdots s_{j_r}$. Fix the total order on $\cR^+$ given by $\beta_1=\alpha_{k_1}$, $\cdots$, $\beta_t=s_{k_1}\cdots s_{k_{t-1}}(\alpha_{k_t})$, $\beta_{t+1}=\bbw_0\alpha_{j_1}$, $\cdots$, $\beta_{k+r}=\bbw_0s_{j_1}\cdots s_{j_{r-1}}(\alpha_{j_r})$. 

For any $\beta\in\cR^+(\bbw_0)$, there exists a unique $1\leq j\leq \ell$ and a unique $\beta_0\in\cR^+(\bs_{i_j})$ such that $\beta=\bs_{i_1}\bs_{i_2}\cdots \bs_{i_{j-1}}(\beta_0)$. Similar to  \eqref{def:Bbeta}, we define 
\begin{align}
\label{def:Bbeta-div}
    B_\beta^{(m)}:=\bT_{i_1}\bT_{i_2}\cdots \bT_{i_{j-1}}(B_{\beta_0}^{(m)}), \qquad \text{ for }\beta\in \cR^+(\bbw_0),
\end{align}
where $B_{\beta_0}^{(m)}$ is the $\imath$divided power defined in Section~\ref{subsec:divid-root-rank1} for $\beta_0\in \cR^+(\bs_{i_j})$. 

Then we 
define $B^{(\ba)}=\prod_{\beta\in\cR^+(\bbw_0)} B_{\beta}^{(a_\beta)}$ for $\ba=(a_\beta)_{\beta}\in\N^{\ell(\bbw_0)}$. The elements $F_\bullet^{(\bc)}$ and $E_\bullet^{(\bd)}$ are defined in the same way as Proposition \ref{prop:PBW-QG}. The following theorem is the main result of this section.

\begin{theorem}
\label{thm:PBWUidot}
Let $\bvs=(\vs_i)_{i\in\I_\circ}$ be parameters such that $\vs_i$ has the form $ \vs_i =q^{a_i}\vs_{\diamond,i},a_i\in\Z$ for any $i\in\I_\circ$.
Let $\bbw_0=\bs_{i_1}\bs_{i_2}\cdots\bs_{i_\ell}$ be a reduced expression for the longest element $\bbw_0$ in $W^\circ$, and $w_\bullet=s_{j_1}s_{j_2}\cdots s_{j_r}$ a reduced expression of the longest element $w_\bullet$ in $W_\bullet$. Then the monomials 
$B^{(\ba)} F_\bullet^{(\bc)} E_\bullet^{(\bd)}\mathbf{1}_\zeta$ ($\ba\in \N^{\ell(\bbw_0)},\bc,\bd\in\N^{|\cR_\bullet^+|}$, $\zeta\in X_\imath$)  
form an $\bar{\A}$-basis of $\dot{\U}^\imath_{\bar{\A}}$.    
\end{theorem}

\begin{proof}
From Theorem \ref{prop:integral-braid} and Lemma \ref{lem:div-integral}, we have $B^{(\ba)}\mathbf{1}_\zeta\in\dot{\U}^{\imath}_{\A}$ for any $\ba$. Then the monomials $B^{(\ba)} F_\bullet^{(\bc)} E_\bullet^{(\bd)}\mathbf{1}_\zeta$ ($\ba\in \N^{\ell(\bbw_0)},\bc,\bd\in\N^{r}$, $\zeta\in X_\imath$) are in $\dot{\U}^{\imath}_{\A}$, and they are linearly independent by Proposition \ref{prop:PBW-Uifiled}. 

Let $\zeta\in X_\imath$. We show that the set of monomials 
\begin{align}
\label{def:cBzeta}
\mathcal{B}(\zeta)=\{B^{(\ba)} F_\bullet^{(\bc)} E_\bullet^{(\bd)}\mathbf{1}_\zeta\mid\ba\in \N^{\ell(\bbw_0)},\bc,\bd\in\N^{|\cR_\bullet^+|}\}
\end{align}
spans $\dot{\U}^\imath_{\bar{\A}}\mathbf{1}_\zeta$. Let $x\in \dot{\U}^\imath_{\bar{\A}}\mathbf{1}_\zeta$. By Proposition~\ref{prop:PBW-Uifiled}, we can write $x=\sum_{b\in \mathcal{B}(\zeta)} c_b b$ for $c_b\in \Q(v)$. By Lemma~\ref{lem:UidotA}, for $\lambda\in X$, we have $x\mathbf{1}_\lambda, b\mathbf{1}_\lambda\in \dot{\U}_{\bar\A}$. Note that $x\mathbf{1}_\lambda\neq0$ if and only if $\bar{\lambda}=\zeta$; hence, we assume $\bar{\lambda}=\zeta$ in the following arguments.

Let $b\in\mathcal{B}(\zeta)$ be a monomial such that $c_b\neq0$ and its weight is maximal among all $b'$ such that $c_b'\neq 0$. We show that $c_b\in \bar{\A}$. Write $b=B^{(\ba)} F_\bullet^{(\bc)}E_\bullet^{(\bd)}\mathbf{1}_\zeta$ for some $\ba,\bc,\bd$. By Lemma~\ref{lem:leading-div}, the leading term of $B^{(\ba)} F_\bullet^{(\bc)} E_\bullet^{(\bd)}\mathbf{1}_\lambda$ is $F^{(\ba)} F_\bullet^{(\bc)} E_\bullet^{(\bd)}\mathbf{1}_\lambda$. i.e.,
\begin{align}
    B^{(\ba)} F_\bullet^{(\bc)} E_\bullet^{(\bd)}\mathbf{1}_{\lambda}
    =F^{(\ba)} F_\bullet^{(\bc)} E_\bullet^{(\bd)}\mathbf{1}_\lambda+\text{ lower terms},
\end{align}
where the lower terms is an $\bar{\A}$-linear combination of elements in the PBW basis of $\dot{\U}_{\bar\A}\mathbf{1}_\lambda$ with strictly lower weights.
On the other hand, we can write $x\mathbf{1}_\lambda$ uniquely as an $\bar{\A}$-linear combination of the PBW basis for $\dot{\U}_{\bar\A}\mathbf{1}_\lambda$ (see Proposition~\ref{prop:PBWUdot}); by the weight reason, the coefficient of $F^{(\ba)}F_\bullet^{(\bc)} E_\bullet^{(\bd)}\mathbf{1}_\lambda$ in this linear combination of $x\mathbf{1}_\lambda$ must equal $c_b$. By the uniqueness, we must have $c_b\in \bar{\A}$ for the chosen $b$. 

We now repeat the above argument for $x-c_b b$. By induction, we conclude that $c_b\in \bar{\A}$ for all $b\in \mathcal{B}(\zeta)$. Therefore, any element in $\dot{\U}^\imath_{\A}$ can be written as a linear combination of $\mathcal{B}(\zeta)$ as desired.
\end{proof}

\begin{corollary}
Let $\bvs=(\vs_i)_{i\in\I_\circ}$ be parameters such that $\vs_i$ has the form $ \vs_i =q^{a_i}\vs_{\diamond,i},a_i\in\Z$ for any $i\in\I_\circ$. Then $\dot{\U}^\imath_{\bar{\A}}$ is a free $\bar{\A}$-module.
\end{corollary}


\subsection{The parabolic subalgebra $\bP$}

Let $\bP=\bP_{\I_\bullet}$ be the $\bF$-subalgebra of $\U_\bF$ generated by $\U_{\I_\bullet}$ and $\U^-$. The algebra $\dot{\bP}=\bigoplus_{\lambda\in X} \bP\mathbf{1}_\lambda$ as a subalgebra of $\dot{\U}$, and its $\A$-form is $\dot{\bP}_\A=\dot{\bP}\cap\dot{\U}_\A$. Similarly, we can define the $\bar{\A}$-form $\dot{\bP}_{\bar{\A}}$.

Note that $\dot{\bP}\cong \dot{\U}/\big(\sum_{x,\lambda\in X}\dot{\U}x\mathbf{1}_\lambda\big)$, where the sum is taken over all homogeneous $x\in\U^+$ whose weights are of the form $|x|=\sum_{i\in\I}a_ii$ with $a_i\neq0$ for some $i\in\I_\circ$. Thus we can view $\dot{\bP}$ as a left $\dot{\U}$-module, and then as a left $\dot{\U}^\imath$-module. 

For $\lambda\in X$, denote by $p_\bvs=p_{\bvs,\lambda}$ the composition map 
\begin{align*}
\dot{\U}^\imath\mathbf{1}_{\ov{\lambda}}\longrightarrow \dot{\U}\mathbf{1}_\lambda\longrightarrow \dot{\U}\mathbf{1}_\lambda/\big(\sum_{x,\lambda\in X}\dot{\U}x\mathbf{1}_\lambda\big)\longrightarrow \dot{\bP}\mathbf{1}_\lambda.
\end{align*}

\begin{lemma}[\text{cf. \cite[Lemma 3.22]{BW18b}}]
Let $\lambda\in X$. The map $p_\bvs=p_{\bvs,\lambda}:\dot{\U}_\bvs\mathbf{1}_{\ov{\lambda}}\rightarrow \dot{\bP}\mathbf{1}_\lambda$
is an isomorphism of left $\dot{\U}^\imath$-modules. Moreover, $p_\bvs$ restricts to an injective homomorphism $p_\bvs:\dot{\U}^\imath_{\bvs,\bar{\A}}\mathbf{1}_{\ov{\lambda}}\rightarrow\dot{\bP}_{\bar{\A}}\mathbf{1}_\lambda$ of $\bar{\A}$-modules.
\end{lemma}

It was shown in \cite[lemma 6.22]{BW18b} that $p_\bvs:\dot{\U}^\imath_{\bvs,\bar{\A}}\mathbf{1}_{\ov{\lambda}}\rightarrow\dot{\bP}_{\bar{\A}}\mathbf{1}_\lambda$ is an $\bar{\A}$-module isomorphism for a class of parameters (see \cite[\S 3.4]{BW18b}); the assumption on parameters made {\em loc. cit.} were used to ensure the existence of the $\imath$-canonical basis, which is essential in their proof. 

Using the integral PBW basis established in Theorem~\ref{thm:PBWUidot}, we can show that $p_\bvs:\dot{\U}^\imath_{\bvs,\bar{\A}}\mathbf{1}_{\ov{\lambda}}\rightarrow\dot{\bP}_{\bar{\A}}\mathbf{1}_\lambda$ is an $\bar{\A}$-module isomorphism for a new class of parameters in the next proposition.

\begin{proposition}
\label{prop:UiPintegral}
Let $\bvs=(\vs_i)_{i\in\I_\circ}$ be parameters such that $\vs_i$ has the form $ \vs_i =q^{a_i}\vs_{\diamond,i},a_i\in\Z$ for any $i\in\I_\circ$. Then $p_\bvs:\dot{\U}^\imath_{\bvs,\bar{\A}}\mathbf{1}_{\ov{\lambda}}\rightarrow\dot{\bP}_{\bar{\A}}\mathbf{1}_\lambda$ is an $\bar{\A}$-module isomorphism.
\end{proposition}

\begin{proof}
It is known that $\dot{\bP}\mathbf{1}_\lambda$ admits a PBW basis of the form $F^{(\ba)}F_\bullet^{(\bc)}E_\bullet^{(\bd)}\mathbf{1}_\lambda$ ($\ba\in \N^{\ell(\bbw_0)},\bc,\bd\in\N^{r}$); cf. Proposition \ref{prop:PBWUdot}. By Lemma \ref{lem:leading-div}, the leading term of $B_\beta^{(m)}$ is $F_\beta^{(m)}$ for any $\beta\in\cR^+(\bbw_0)$, and then the leading term of $B^{(\ba)}$ is $F^{(\ba)}$ for $\ba\in \N^{\ell(\bbw_0)}$. It follows that 
\begin{align}
\label{eq:pbvs}
    p_\bvs(B^{(\ba)} F_\bullet^{(\bc)} E_\bullet^{(\bd)}\mathbf{1}_{\ov{\lambda}})=F^{(\ba)} F_\bullet^{(\bc)} E_\bullet^{(\bd)}\mathbf{1}_\lambda+\text{ lower terms}.
\end{align}
Here the ``lower terms'' are linear combinations of elements of the form 
$F^{(\ba')}F_\bullet^{(\bc')}E_\bullet^{(\bd')}\mathbf{1}_\lambda$, where $\sum_{\beta\in\cR^+(\bbw_0)} a'_\beta<\sum_{\beta\in\cR^+(\bbw_0)} a_\beta$. 

Let $\mathcal{B}(\ov{\lambda})$ be the set defined in \eqref{def:cBzeta}. By Theorem~\ref{thm:PBWUidot}, $\mathcal{B}(\ov{\lambda})$ forms a basis for $\dot{\U}^\imath_{\bvs,\bar{\A}}\mathbf{1}_{\ov{\lambda}}$. By \eqref{eq:pbvs}, the $p_\bvs$-image of $\mathcal{B}(\ov{\lambda})$ forms a basis for $\dot{\bP}\mathbf{1}_\lambda$. Since $p_\bvs$ sends the basis $\mathcal{B}(\ov{\lambda})$ of $\dot{\U}^\imath_{\bvs,\bar{\A}}\mathbf{1}_{\ov{\lambda}}$ to a basis of $\dot{\bP}_{\bar{\A}}\mathbf{1}_\lambda$, the map $p_\bvs$ is an isomorphism of $\bar{\A}$-modules.
\end{proof}

\end{document}